\documentclass{amsart}
\usepackage{bbm}
\usepackage{amsmath, amsthm, amssymb}
\usepackage{verbatim}

\newtheorem{theorem}{Theorem}[section]
\newtheorem{lemma}[theorem]{Lemma}
\newtheorem{proposition}[theorem]{Proposition}

\theoremstyle{definition}
\newtheorem{definition}[theorem]{Definition}

\theoremstyle{remark}

\numberwithin{equation}{section}

\usepackage{biblatex}
\begin{document}

\title{Multiple Dirichlet Series Associated With Quadrics}

\author{Jun Wen}
% \address{Quant Trader, New York}
% \curraddr{Department of Mathematics and Statistics,
% Case Western Reserve University, Cleveland, Ohio 43403}
%\email{jun.wen@untlcapital.com}
% \thanks{The first author was supported in part by NSF Grant \#000000.}


%    General info
% \subjclass[2020]{Primary 54C40, 14E20; Secondary 46E25, 20C20}

% \date{January 1, 2001 and, in revised form, June 22, 2001.}

% \dedicatory{This paper is dedicated to our advisors.}

% \keywords{Differential geometry, algebraic geometry}

\begin{abstract}
We define a multiple Dirichlet series associated with quadrics which is the zero locus of a quadratic form. 
This multiple Dirichlet series is linked to a Shintani zeta function associated with 
a prehomogeneous vector space. To obtain the functional equations we construct a filtration
of the quadratic space and define the parabolic group actions, and then apply a non-abelian Poisson summation formula 
which sums over all lower dimensional quadrics along with the original quadrics.  We show the group  of functional equations is isomorphic to a finite Weyl group of type $A_3$. 
\end{abstract}

\maketitle

%% The correct journal style for \specialsection is all uppercase; a known bug
%% in amsart.cls prevents this, so input must be uppercase until it is fixed.
%\specialsection*{This is a Special Section Head}
% \specialsection*{THIS IS A SPECIAL SECTION HEAD}
% This is an example of a special section head%
%%%%%%%%%%%%%%%%%%%%%%%%%%%%%%%%%%%%%%%%%%%%%%%%%%%%%%%%%%%%%%%%%%%%%%%%
% \footnote{Here is an example of a footnote. Notice that this footnote
% text is running on so that it can stand as an example of how a footnote
% with separate paragraphs should be written.
% \par
% And here is the beginning of the second paragraph.}%
%%%%%%%%%%%%%%%%%%%%%%%%%%%%%%%%%%%%%%%%%%%%%%%%%%%%%%%%%%%%%%%%%%%%%%%%

\section{Introduction}

We restrict to the rational number filed $F = \mathbb{Q}$  in the paper.
Let $V_{\mathrm{bqf}}(F) = \{x(u, v) = x_1u^2 + x_2uv + x_3v^2 : x_1, x_2, x_3 \in F \}$ be the space of
binary quadratic forms over $F$. The group $\mathrm{GL}_2(F)$ acts on this vector space via
$\rho(g)x(u, v) = x(au + c v, bu + dv)$ for $g = \left( \begin{matrix}
    a & b \\
    c & d
\end{matrix} \right)$, whose  matrix representation is given by
\[
\rho(g) = \left( \begin{matrix}
a^2 &ab &b^2 \\
2ac &ad + bd &2bd\\
c^2 &cd & d^2
\end{matrix}
\right).
\]
Denote by $B_2^{+}(F)$ the subgroup of lower triangular matrices with positive diagonal entries. The representation $(B_2^{+}(F), V_{\mathrm{bqf}}(F))$ is called the prehomogeneous vector space, which is acted on by a reductive algebraic group such that the orbits are Zariski open sets. 
Under this action there are two polynomial invariants
\[
P_1(x) = x_1 \quad \mathrm{and} \quad P(x) = \mathrm{Disc}(x) = x_2^2 - 4 x_1 x_3.
\]
Let $\chi_1$ and $\chi$ be the characters of $B^{+}_2(F)$ defined by
\[
\chi_1(g) = a^2 \quad \mathrm{and} \quad \chi(g) = \mathrm{det}(g)^2.
\]
We have
\[
P_1(g\cdot x) = \chi_1(g) x_1 \quad \mathrm{and} \quad P(g \cdot x) = \chi(g) P(x).
\]
There is a bilinear pairing on $V_{\mathrm{bqf}}(F)$
\[
\langle x, y\rangle = x_1y_3 - 2^{-1} x_2 y_2 +x_3 y_1.
\]
Let $J$ be the matrix of the pairing $\langle\cdot, \cdot \rangle$ such that $\langle x, y \rangle = x^tJy$ , then 
\[
J = \left( \begin{matrix}
0 &0 &1 \\
0 &-\frac{1}{2} &0\\
1 &0 & 0
\end{matrix}
\right).
\]
Let $Q$ be the corresponding quadratic form of the inner product given by $Q(x) = \frac{x^t J x}{2}$.
The Fourier transform of a Schwartz function $f$ on $V_{\mathrm{bqf}}(\mathbb{R})$ is given by $f(\xi) = \int_{V_{\mathrm{bqf}(\mathbb{R}})} f(x) e^{2 \pi \sqrt{-1} \langle x, \xi \rangle} dx$.

The double Dirichlet series $\xi(s, w)$ and $\xi^{\ast}(s, w)$ are functions with two complex variables $s = (s, w) \in \mathbb{C}^2$, and they are defined by the following expressions:
\begin{equation} \label{eq:double series}
\begin{split}
& \xi_i(s, w) = 2^{-1} \sum_{n, m = 1}^{\infty} A(4m, (-1)^{i-1}n) m^{-s} n^{-w} \\
& \xi_i^{\ast}(s, w) = \sum_{n, m = 1}^{\infty} A(m, (-1)^{i-1}n) m^{-s} (4n)^{-w}
\end{split}
\end{equation}
where $A(m, n)$ denotes the number of distinct solutions to the quadratic congruence equation $x^2 \equiv n \ (\mathrm{mod} \ m)$ for a given pair of  integers $m$ and $n$.

The zeta function associated with 
the prehomogeneous vector space was introduced by Sato and Shintani in \cite{sato1974zeta}.
Shintani  \cite{shintani1975zeta}  derived the functional equations of double Dirichlet series  by studying the zeta functions associated with the lattice of integral binary quadratic forms. Let $L = V_{\mathrm{bqf}}(\mathbb{Z})$ be such lattice with the dual lattice $L^{\ast}$.  Let $\Sigma_{\mathrm{sig}} = \{x: P_1(x) P(x) =0 \}$ be the singular subset. Let $L' = L - L \cap \Sigma_{\mathrm{sig}}, L^{\ast'} = L^{\ast} - L^{\ast} \cap \Sigma_{\mathrm{sig}}$. The Shintani zeta functions are defined as
\begin{equation*}
\begin{split}
Z(f, s, w) &:= \int_{B_2^{+}(\mathbb{R}) / B^{+}_2(\mathbb{Z})} \chi_1(h)^{s} \chi(h)^{w}  \sum_{x \in  L'} f(h \cdot x) dh \\
Z^{\ast}(f, s, w) &:= \int_{B_2^{+}(\mathbb{R}) / B^{+}_2(\mathbb{Z})} \chi_1(h)^{s} \chi(h)^{w}  \sum_{x \in  L^{\ast '}} f(h \cdot x) dh
\end{split}
\end{equation*}
where $f$ is the Schwartz function of $V_{\mathrm{bqf}}(\mathbb{R})$ . These zeta functions are originally absolutely convergent for $\mathrm{Re}(s), \mathrm{Re}(w) > 1$. Shintani showed they satisfy the following functional equations
$$Z(f, s, w) \mapsto Z(f, 1-s, s+w - 1/2) \quad   \mathrm{and} \quad  Z(f, s, w) \mapsto Z^{\ast}(f^{\ast}, s, 3/2 - s - w)$$
and therefore have analytic continuations to the complex plane $\mathbb{C}^2$. In fact, 
the functional equation relating the values at $(s, w)$ and $(1-s, s + w - 1/2)$  comes from the functional equation of Eisenstein series of $\mathrm{SL}_2$; the second functional equation relating the values at $(s, w)$ and $(s, 3/2 - s - w)$ comes from a Poisson summation formula applied to the Schwartz function associated to the prehomogeneous vector space. The zeta functions are related to the double Dirichlet series through the orbital integrals
\[
\Phi_{i}(f, s, w) := \int_{V_i(\mathbb{R})} f(x)   |P_1(x)|^{s} |P(x)|^{w} dx
\]
where $V_i(\mathbb{R}): = \{x \in  V_{\mathrm{bqf}}(\mathbb{R}), (-1)^{i-1} P(x) > 0   \}$.
The relation between Shintani double zeta functions and  $A_2$  Weyl group multiple Dirichlet series is obtained through the 
Mellin transforms of metaplectic Eisenstein series on $\mathrm{GL}_2$ in \cite{diamantis2014converse}. The adelic version of this relation over any number
field has been studied in a recent work \cite{kim2022shintani}.

In this paper, we study a zeta function in three variables defined by
\begin{equation} \label{eq:A3 multiple Dirichlet series}
\begin{split}
 \xi_{\pm}(s_1, s_2, w) := \sum_{\pm D > 0} \frac{1}{|D|^{w}} \sum_{n, m \geq 1} \frac{\sum_{\substack{d \geq 1, d^2|D \\ d|m, d|n}} d \cdot A(\frac{4m}{d}, \frac{D}{d^2}) A(\frac{4n}{d}, \frac{D}{d^2})}{m^{s_1} n^{s_2}}.
\end{split}
\end{equation}
This zeta function comes from the counting function of integral orbits of a spherical variety. For this, take two copies of vector spaces of binary quadratic forms and define $V = \oplus_{i=1}^2 V_{\mathrm{bqf},i}$. For each copy, let $Q_i$ be the binary quadratic form on $V_{\mathrm{bqf},i}$. Equip the vector space $V$ with a nondegenerate quadratic from 
\[
Q(v) = Q(x, y) = Q_1(x) - Q_2(y), \quad x \in V_{\mathrm{bqf},1}, y \in V_{\mathrm{bqf},2}.
\]
 The zero locus  $X \subset V$ is defined as
 \[
X(F) := \{ v \in V(F): Q(v) = 0 \}.
\]
Let $B^{+}$ be the subgroup of pairs $(g_1, g_2) \in B_2^{+} \times B_2^{+}$ such that $\mathrm{det}(g_1) = \mathrm{det}(g_2)$. The action of $B^{+}$ on $V$ has three relative invariants, namely
\[
P_1(v) = P_1(x), \quad P_2(v) = P_1(y),   \quad P(v) = \mathrm{Disc}(x).
\]
The characters $\chi_i$  $ (i = 1, 2)$ and $\chi$ are given by
\[
\chi_1(g) = \chi_1(g_1), \quad \chi_2(g) = \chi_2(g_2),   \quad \chi(g) = \chi(g_1) = \chi(g_2).
\]
Denote by $\Sigma_{\mathrm{sig}} = \{v\in X: P_1(v) P_2(v) P(v) = 0 \}$ the singular subset of $X$.  Let $L = X(\mathbb{Z})$ be the integral points of $X$. Set $L' =L - L \cap \Sigma_{\mathrm{sig}}$.
The representation $(B^{+}(F), V(F))$ is not a prehomogeneous vector space. However, we can still define a zeta function associated with the non-singular integral points of $X$ as follows
\[
Z(f, s_1, s_2, w) := \int_{B^{+}(\mathbb{R}) / B^{+}(\mathbb{Z})} \chi_1(h)^{s_1} \chi_2(h)^{s_2} \chi(h)^{w} \sum_{n=1}^{\infty} \sum_{v \in n L'} n \cdot f(h \cdot v) dh.
\]
In order to derive the functional equations of \eqref{eq:A3 multiple Dirichlet series}, we require three ingredients: 
\begin{itemize}
\item The space of Schwartz functions $\mathcal{S}(X)$ on $X$. 
\item Fourier transform $\mathcal{F}_X$ of Schwartz space $\mathcal{S}(X)$.
\item Poisson summation formula relating $\sum_{v \in L}f(v)$ and $\sum_{v \in L} \mathcal{F}_X(f)(v)$.  
\end{itemize}
The Poisson summation formula is well understood on the vector space and has been applied to many cases in number theory including establishing the functional equations of zeta functions associated with prehomogeneous vector spaces.  
The work of \cite{braverman2010gamma} proposed the Poisson summation conjecture for certain affine spherical varieties which are now called  Braverman–Kazhdan space. The conjecture has been investigated and extended to other spherical varieties in \cite{braverman1999schwartz,  braverman2002normalized, lafforgue2014noyaux, ngo2020hankel, sakellaridis2012spherical}.
The Poisson summation formula and the harmonic analysis on the Braverman–Kazhdan spaces has been developed in \cite{getz2021refined, getz2020summation, getz2021harmonic}.
In \cite{getz2019summation}, the Poisson summation formula was obtained as a first example for a spherical variety that is not a Braverman–Kazhdan space.  Recently, this formula has been generalized for the zero locus of a quadratic form in \cite{getz2023summation}. Built on this Poisson summation formula, we can investigate the zeta functions associated with spaces of quadrics as a natural generalization of Shintani zata functions associated with prehomogeneous vector spaces. The following theorem is the main result of this paper:
\begin{theorem}
For $\mathrm{Re}(s_1), \mathrm{Re}(s_2) > 1$,  the zeta functions $Z(f, s_1, s_2, w)$ and $Z(\mathcal{F}_X(f), s_1, s_2, w)$ have analytic continuations in the whole $w$-plane and satisfy the following functional equation
\begin{equation}
\label{eq:main functional equation on w}
\begin{split}
Z( f, s_1, s_2, w)  &=  Z(\mathcal{F}_X(f), s_1, s_2, 2-s_1-s_2-w).
\end{split}
\end{equation}
\end{theorem}
We also give the functional equations of zeta function $Z(f, s_1, s_2, w)$ relating the values at $(s_i, w)$ and  $(1-s_i, s_i + w - \frac{1}{2})$ for $i = 1, 2$. Therefore by Bochner's convexity theorem, the multiple Dirichlet series $\xi_{\pm}(s_1, s_2, w)$ have analytic continuations as meromorphic functions in $\mathbb{C}^3$.
\begin{theorem}
The multiple Dirichlet series $\xi_{\pm}(s_1, s_2, w)$
have analytic continuations as meromorphic functions in $\mathbb{C}^3$. In addition to the functional equation \eqref{eq:main functional equation on w},  the functions
\begin{equation}
\begin{split}
&\prod_{i=1}^2 (2 \pi)^{-s_i} \left(\sin \frac{\pi s_i}{2} \right)^{-1}  \Gamma(s_i)\zeta(2 s_i) \xi_{+}(s_1, s_2, w),\\
&\prod_{i=1}^2  (2 \pi)^{-s_i} \Gamma(s_i)\zeta(2 s_i) \xi_{-}(s_1, s_2, w)
\end{split}
\end{equation}
are invariant under the transformations $(s_i, w) \to (1-s_i, s_i + w - \frac{1}{2})$ for $i = 1, 2$.
\end{theorem}
Before finishing the introduction we want to mention its connection to Weyl group multiple Dirichlet series. Let $G' = B^{+}_2 \times B^{+}_2 \times \mathrm{SL}_2$ and $V' = F^2 \otimes F^2 \otimes F^2$. The representation $(G' , V')$ is a $D_4$-type prehomogeneous vector space of split case. The global zeta function for the non-split case when $V'$ is a pair of simple algebras has been studied in \cite{taniguchi2007zeta}. But the analysis for split case becomes complicated. In \cite{bhargava2004higher}, Bhargava discovered the law of group composition for integral cubes $\mathbb{Z}^2 \otimes \mathbb{Z}^2 \otimes \mathbb{Z}^2$. In \cite{wen2013bhargava}, the author studied the multiple Dirichlet series associated with the prehomogeneous vector space $(G', V')$ and obtained \eqref{eq:A3 multiple Dirichlet series}. Its relation to $A_3$ Weyl group multiple Dirichlet series was obtained by matching its $p$-part with that given in \cite[Example 3.7]{chinta2007weyl}. The three generators of functional equations are relating values between $(s_i, w)$ and $(1-s_i, s_i + w - \frac{1}{2})$ for $i = 1, 2$, and values between $(s_1, s_2, w)$ and $(s_1+w - \frac{1}{2}, s_2+w - \frac{1}{2}, 1 - w)$. For general description of those involutions, see \cite[Section 5]{chinta2007weyl}. It is easy to check that the functional equation relating $(s_1, s_2, w)$ and $(s_1, s_2, 2 - s_1 - s_2 - w)$ is generated by those three involutions. The first two functional equations can also be obtained from the functional equations of quadratic Dirichlet $L$-functions. Thus the main result of this paper is establishing the third functional equation implied by \eqref{eq:main functional equation on w}.

\section{Poisson summation formula on singular varieties}

\subsection{Filtration of quadratic spaces}
For the quadratic form $Q$ associated with the quadratic space $V$, let $J$ denote the bilinear form which is defined by $Q(x) = \frac{x^t J x}{2}$. Identify $V(F)$ with $F^6 = \{(x_1, x_2, x_3, x_4, x_5, x_6)\}$.  Then its matrix form is given by

\begin{equation*}
\begin{split}
J = \left( \begin{matrix}
0 &0 &1 &0 & 0 &0\\
0 &-\frac{1}{2} &0 &0 &0 &0\\
1 &0 & 0 &0 &0 &0\\
0 & 0 &0 &0 &0 &-1\\
0 & 0 &0 &0 &\frac{1}{2} &0\\
0 & 0 &0 &-1 &0 & 0
\end{matrix}
\right).
\end{split}
\end{equation*}
The filtration of $V$ is defined as a sequence of vector spaces $V_0 = \{0\} \subset V_1 \subset V_2  \subset V_3 = V$, such that the subspaces $V_1$ and $V_2$ are given by
 \begin{equation*}
\begin{split}
&V_2(F) := \{(x_1, x_2, x_3, x_5): x_i \in F\},  \quad  V_1(F) := \{(x_2, x_5): x_i \in F\}.
\end{split}
\end{equation*}
Define
\begin{equation*}
\begin{split}
V_4(F) := V_3(F) \oplus F^2.
\end{split}
\end{equation*}
For $i > i'$, identity $V_{i'}$ with the subspace of $V_i$ as
\begin{equation*}
\begin{split}
& V_{i'}(F) \oplus \{0\}^{2i - 2i'} \subset V_i(F).
\end{split}
\end{equation*}
Let $Q_i$ be the quadratic form associated with each subspace $V_i(F)$. Let $J_i$ denote the corresponding bilinear form. Then $Q_i(x) = \frac{x^t J_i x}{2}$. Given such filtration each subspace $V_i(F)$ can be identified with the direct sum $V_{i-1}(F) \oplus F^2$. The restriction of the bilinear form $J$ on the two dimensional subspace $F^2$ of $V_i$ is given by
\begin{equation*}
\begin{split}
J|_{F^2 \subset V_1} = \left( \begin{matrix}
-\frac{1}{2} &0\\
0 &\frac{1}{2}
\end{matrix}
\right), \quad J|_{F^2 \subset V_2} = \left( \begin{matrix}
0 &1\\
1 &0
\end{matrix}
\right), \quad J|_{F^2 \subset V_3} = \left( \begin{matrix}
0 &-1\\
-1 &0
\end{matrix}
\right).
\end{split}
\end{equation*}
For $V_1$,  after changing variables $(x_2, x_5) = (u-v, u+v)$, the matrix form of $J_1$ is given by
\begin{equation*}
\begin{split}
J_1 = \left( \begin{matrix}
0 &1\\
1 &0
\end{matrix}
\right).
\end{split}
\end{equation*}
For $V_2 = V_1 \oplus F^2$. The matrix form of $J_2$ is given by
\begin{equation*}
\begin{split}
J_2 = \left( \begin{matrix}
J_1 &\\
  &J_1
\end{matrix}
\right).
\end{split}
\end{equation*}
For $V_3 = V_2 \oplus F^2$,  by variable changes $(x_4, x_6) = (u, -v)$ on the $F^2$ part,  the quadratic form $Q_3$ is defined such that
\begin{equation*}
\begin{split}
J_3 = \left( \begin{matrix}
J_1 & &\\
 &J_1 &\\
 & &J_1
\end{matrix}
\right).
\end{split}
\end{equation*}
For $V_4(F) = V_3(F) \oplus F^2$ the corresponding quadratic form $Q_4$ is defined such that 
\begin{equation*}
\begin{split}
J_4 = \left( \begin{matrix}
J_1 & & &\\
  &J_1& &\\
  & &J_1& \\
  & & &J_1
\end{matrix}
\right).
\end{split}
\end{equation*}
Define the zero locus of each quadratic space to be
\begin{equation*}
\begin{split}
X_i := \{u \in V_i: Q_i(u) = 0\}
\end{split} \quad \text{($1 \leq i \leq 3$)}
\end{equation*}
and let $X_i^0 = X_i - \{0\}$. 

\subsection{Filtration of parabolic group actions}
Denote by $\mathrm{GO}_{Q_i}$ the orthogonal similitude group of the quadratic space $(V_i, Q_i)$ for $1 \leq i \leq 4$. Let 
\begin{equation*}
\begin{split}
\lambda:  \mathrm{GO}_{Q_i} &\to \mathbb{G}_m
\end{split}
\end{equation*}
be the similitude norm. Define the natural subgroup embedding with respect to the subspace identification $V_i(F) = V_{i-1}(F) \oplus F^2$ as
\begin{equation}
\label{eq:group embedding}
\begin{split}
\mathrm{GO}_{Q_i} & \to \mathrm{GO}_{Q_{i+1}} \\
h & \mapsto \left( \begin{matrix}
    h &           &\\
      &\lambda (h)& \\
      &           & 1\\  
\end{matrix}  \right). \\
\end{split}
\end{equation}
Let $B_2$ denote the subgroup of $\mathrm{SL}_2$ consisting of lower triangular matrices. Let $B$ denote the subgroup of $\mathrm{GO}_Q$ consisting of pairs $(g_1, g_2) \in B_2 \times B_2$ such that $\mathrm{det}(g_1) = \mathrm{det}(g_2)$. Then the similitude norm of group $B$ is given by
\begin{equation*}
\begin{split}
\lambda: B &\to \mathbb{G}_m \\
g &\mapsto \left(\mathrm{deg}(g_1)\right)^2.
\end{split}
\end{equation*}
The ring of adeles and the group of ideles are denoted by $\mathbb{A}$ and $\mathbb{A}^{\times}$, respectively. Let $N\leq B$ be the unipotent subgroup and $T \leq B$ be the subgroup of diagonal matrices of $B$. Then define
\begin{equation*}
\begin{split}
N &= \left\{ \left( \begin{matrix}
    n_u & \\
        &n_v
\end{matrix}  \right): u, v \in \mathbb{A}_F \right\}, \quad n_u = \left( \begin{matrix}
    1 & \\
    u &1
\end{matrix}  \right), \\
A &= \left\{ \left( \begin{matrix}
    a_t & \\
        &a_{s}
\end{matrix}  \right): t, s \in \mathbb{A}_F^{\times} \right\}, \quad a_t = \left( \begin{matrix}
    t &0 \\
    0 &\frac{1}{t}
\end{matrix}  \right), \\
\Lambda &= \left\{ \left( \begin{matrix}
    d_{\mu} & \\
        &d_{\mu}
\end{matrix}  \right): \mu \in \mathbb{A}_F^{\times} \right\}, \quad d_{\mu} = \left( \begin{matrix}
    \mu &0 \\
    0 &\mu
\end{matrix}  \right),
\end{split}
\end{equation*}
such that $T = A \Lambda$. For $h \in B$, it can be expressed as
\begin{equation} \label{eq:h form}
\begin{split}
h = \left( \begin{matrix}
    a_t & \\
        &a_{s}
\end{matrix}  \right) \left( \begin{matrix}
    n_u & \\
        &n_v
\end{matrix}  \right)
\left( \begin{matrix}
    d_{\mu} & \\
        &d_{\mu}
\end{matrix}  \right).
\end{split}
\end{equation}
Define a subgroup $B' \leq B$ as $\{h \in B: h = \left( \begin{matrix}
    a_t & \\
        &1
\end{matrix}  \right) \left( \begin{matrix}
    n_u & \\
        &1
\end{matrix}  \right)
\left( \begin{matrix}
    d_{\mu} & \\
        &d_{\mu}
\end{matrix}  \right)     \}$.
\begin{lemma}
With respect to the filtration of quadratic space and the chosen coordinates, the group embedding 
$B  \to \mathrm{GO}_{Q_3}$ is given by
\begin{equation*}
\begin{split}
\iota  \left( \begin{matrix}
    d_{\mu} & \\
        &d_{\mu}
\end{matrix}  \right) &= 
 \left( \begin{matrix}
\mu^2 I_2 & &\\
 & \mu^2 I_2 &\\
    & &\mu^2 I_2
\end{matrix}
\right), \\
\iota  \left( \begin{matrix}
    a_t & \\
        &a_{s}
\end{matrix}  \right) &= 
 \left( \begin{matrix}
   1 & & & & &\\
    &1 & & & &\\
    & &t^{2}   & & &\\
    & & &t^{-2} & &\\
    & & & &s^{2} & \\
    & & & & &s^{-2}
\end{matrix}
\right), \\
\iota  \left( \begin{matrix}
    n_u & \\
        &n_v
\end{matrix}  \right)   &= 
 \left( \begin{matrix}
   1 &0 &u &0 &v &0\\
   0 &1 &-u &0 &v &0\\
   0 &0 &1  &0 &0 &0\\
   u &-u &u^2 &1 &0 &0\\
   0 &0 &0 &0 &1 &0 \\
   -v &-v &0 &0 &-v^2 &1
\end{matrix}
\right).
\end{split}
\end{equation*}
\end{lemma}

Note that $B$ does not act invariant on the subspace $V_i$  for $i = 1, 2$, $B'$ acts invariant on $V_2$ but not invariant on the subspace $V_1$. We give the following definition for those actions.
\begin{definition}
With respect to the filtration $V_0 = \{0\} \subset V_1 \subset V_2  \subset V_3 \subset V_4$, let 
\begin{enumerate}
\item    $\iota_3: B \to \mathrm{GO}_{Q_{4}}$ denote the embedding $\iota: B \to \mathrm{GO}_{Q_{3}}$ followed by the group embedding  $\mathrm{GO}_{Q_{3}} \to \mathrm{GO}_{Q_{4}}$ given by \eqref{eq:group embedding}.
\item $\iota_2$ denote the embedding $\iota: B \to \mathrm{GO}_{Q_{3}}$.
\item $\iota_1$ denote the restriction of $\iota$ on $B'$ . Because $B' \cdot V_1 \subset V_2$. This gives the embedding $\iota_1: B' \to \mathrm{GO}_{Q_{2}}$.
\end{enumerate}
\end{definition}
It is easy to check the following expressions
\begin{equation*}
\begin{split}
\iota_1  \left( \begin{matrix}
    n_u & \\
        &1
\end{matrix}  \right)   &= 
 \left( \begin{matrix}
   1 &0 &u &0\\
   0 &1 &-u &0\\
   0 &0 &1  &0\\
   u &-u &u^2 &1
\end{matrix}
\right) , \quad 
\iota_2  \left( \begin{matrix}
    1 & \\
        &n_v
\end{matrix}  \right)   = 
 \left( \begin{matrix}
   1 &0 &0 &0 &v &0\\
   0 &1 &0 &0 &v &0\\
   0 &0 &1  &0 &0 &0\\
   0 &0 &0 &1 &0 &0\\
   0 &0 &0 &0 &1 &0 \\
   -v &-v &0 &0 &-v^2 &1
\end{matrix}
\right).
\end{split}
\end{equation*}
Both of them take the form
\begin{equation}
\label{eq:u form}
\begin{split}
u(x)  &= 
 \left( \begin{matrix}
   I_{V_i}  &J_i x &0\\
       0    &1 & 0\\
   -x^t    &-Q_i(x) & 1
\end{matrix}
\right).
\end{split}
\end{equation}
The subgroup defined by $\{ u(x), x \in F \}$ is the unipotent subgroup of the maximal parabolic subgroup of $\mathrm{GO}_{Q_{i+1}}(F)$.

\subsection{Weil representation and its extensions}
Let 
\begin{equation*}
\begin{split}
\mathcal{S}(V_i(\mathbb{A}_F)) := \mathcal{S}(V_i(F_{\infty})) \otimes C_{c}^{\infty}(V_i(\mathbb{A}_F^{\infty}))
\end{split}
\end{equation*}
be the space of Schwartz functions on each $V_i(\mathbb{A}_F)$. Let $R$ denote the right representation of the orthogonal group $\mathrm{GO}_{Q_i}(\mathbb{A}_F)$ on the Schwartz space $\mathcal{S}(V_i(\mathbb{A}_F))$
\begin{equation*}
\begin{split}
R: \mathrm{GO}_{Q_i}(\mathbb{A}_F) \times \mathcal{S}(V_i(\mathbb{A}_F)) &\to \mathcal{S}(V_i(\mathbb{A}_F)) \\
(h, f) & \mapsto \left( v \to f(h \cdot v) \right).
\end{split}
\end{equation*}
Let $\rho_i$ be the Weil representation on the $\mathrm{SL}_2$ factor, 
\begin{equation*}
\begin{split}
\rho_i: \mathrm{SL}_2(\mathbb{A}_F) \times \mathcal{S}(V_i(\mathbb{A}_F)) \to \mathcal{S}(V_i(\mathbb{A}_F)).
\end{split}
\end{equation*}
At a local place $v$ of $F$, the Weil representation of $\mathrm{SL}_2(F_v)$ is given by
\begin{enumerate}
\item $\rho_i\left( \begin{matrix}
     a& \\
      &a^{-1} 
\end{matrix}  \right) f(u) = \chi_{Q_i}(a) |a|^{\mathrm{dim}_{F_v}V_i/2} f(au) \ \mathrm{for} \ a \in F_v^{\times}$. 
\item $\rho_i\left( \begin{matrix}
    1 &t \\
      &1 
\end{matrix}  \right) f(u) =  \psi_v(t Q_i(u)) f(u)$.
\item $\rho_i\left( \begin{matrix}
     &1 \\
     -1 & 
\end{matrix}  \right) f(u) = \gamma(Q_i)\int_{V_i} f(x) \psi_v(u^t J_i x) dx$ .

\end{enumerate}
The character $\chi_{Q_i}$  is given by Hilbert symbol
\begin{equation*}
\begin{split}
\chi_{Q_i}(a) :=  (a, (-1)^{\mathrm{dim}_{F_v}V_i/2} \mathrm{det}(J_i)),
\end{split}
\end{equation*}
which is independent of each quadratic subspace $V_i$. 
Let $dx_{\infty}$ denote the usual Lebesgue measure on $\mathbb{R}$ and $dx_v$ denote the Haar measure on $F_v$ normalized by
$\int_{\mathcal{O}_v} dx_v = 1$. 
The additive character $\psi: F\backslash \mathbb{A}_F \to \mathbb{C}^{\times}$ is chosen such that over local field $F_v$  the pairing 
$(u, x) \mapsto \psi_v(u^t J_i x)$ makes the $dx$ a self-dual measure. In our case, they are defined by
\begin{equation*}
\begin{split}
\psi_v(x) := \begin{cases}
\exp \left(2\pi \sqrt{-1} x\right) \quad & \text{if $F_v = \mathbb{R}$}, \\
\exp \left(-2\pi \sqrt{-1} [x]_p\right)     & \text{if $F_v = \mathbb{Q}_p$}.
\end{cases}
\end{split}
\end{equation*}

The representation of the orthogonal group $ \mathrm{GO}_{Q_i}(\mathbb{A}_F)$ and the Weil representation of $\mathrm{SL}_2(\mathbb{A}_F)$ on the Schwartz space $\mathcal{S}(V_i(\mathbb{A}_F))$ do not commute. In fact, we have the following \cite[Lemma 3.1]{getz2019summation}
\begin{lemma}
\begin{equation} \label{eq:twisted g}
\begin{split}
R(h) \circ \rho_i(g) = \rho_i(g^h) \circ R(h),
\end{split}
\end{equation}
where
\begin{equation*}
\begin{split}
g^h = \left( \begin{matrix}
     1& \\
      &\lambda(h)^{-1} 
\end{matrix}  \right) g  \left( \begin{matrix}
     1& \\
      &\lambda(h)
\end{matrix}  \right).
\end{split}
\end{equation*}
\end{lemma}
The Weil representation $\rho_i$ of $\mathrm{SL}_2(\mathbb{A}_F)$ defined on the Schwartz space  $\mathcal{S}(V_i(\mathbb{A}_F))$  can be extended to the representation on the Schwartz space 
$\mathcal{S}(V_i(\mathbb{A}_F) \oplus \mathbb{A}_F^2)$ by
\begin{equation*}
\begin{split}
r_i(g) : \mathcal{S}(V_i(\mathbb{A}_F)) \otimes \mathcal{S}(\mathbb{A}_F^2) &\to \mathcal{S}(V_i(\mathbb{A}_F)) \otimes \mathcal{S}(\mathbb{A}_F^2) \\ 
f_1 \otimes f_2 &\mapsto \left( (\xi, v) \mapsto \rho_i(g)f_1(\xi) f_2(g^t v) \right).
\end{split}
\end{equation*}
Define a partial Fourier transform of $\mathcal{S}(V_i(\mathbb{A}_F) \oplus \mathbb{A}_F^2)$ in the second variable of $\mathbb{A}^2_F$ by
\begin{equation}
\label{eq:partial Fourier transform}
\begin{split}
\mathcal{F}_{2, i}: \mathcal{S}(V_{i+1}(\mathbb{A}_F)) &\to \mathcal{S}(V_i(\mathbb{A}_F) \oplus \mathbb{A}_F^2) \\
f &\mapsto \left( (\xi, u_1, u_2) \mapsto \int_{\mathbb{A}_F} f(\xi, u_1, x) \psi(u_2x) dx \right).
\end{split}
\end{equation}
Then we have the following property \cite[Lemma 4.2]{getz2023summation}
\begin{lemma}
\begin{equation} \label{eq: F2 commutative}
\begin{split}
\mathcal{F}_{2, i} \circ \rho_{i+1}(g)  = r_i(g) \circ \mathcal{F}_{2, i}, \quad g \in \mathrm{SL}_2(\mathbb{A}_F).
\end{split}
\end{equation}
\end{lemma}
Combining formulas \eqref{eq:twisted g} and \eqref{eq: F2 commutative}, we have
\begin{equation*}
\begin{split}
(\mathcal{F}_{2, i} \circ R(h) \circ \mathcal{F}_{2, i}^{-1}) \circ r_{i}(g)  = r_i(g^h) \circ (\mathcal{F}_{2, i} \circ R(h) \circ \mathcal{F}_{2, i}^{-1})
\end{split}
\end{equation*}
for $g \in \mathrm{SL}_2(\mathbb{A}_F)$, $h \in \mathrm{GO}_{Q_{i+1}}(\mathbb{A}_F)$.  Therefore if we define 
 \begin{equation*}
\begin{split}
\sigma_i(h) := \mathcal{F}_{2, i} \circ R(h) \circ \mathcal{F}_{2, i}^{-1},
\end{split}
\end{equation*}
as the representation of  orthogonal  $\mathrm{GO}_{Q_{i+1}}(\mathbb{A}_F)$ on the Schwartz space $\mathcal{S}(V_i(\mathbb{A}_F) \oplus \mathbb{A}_F^2)$, it satisfies
\begin{lemma}
\label{lemma:sigma r = r sigma}
\begin{equation}
\label{eq:sigma r = r sigma}
\begin{split}
\sigma_i(h) \circ r_{i}(g)  = r_i(g^{h}) \circ \sigma_i(h), \quad g \in \mathrm{SL}_2(\mathbb{A}_F), h \in \mathrm{GO}_{Q_{i+1}}(\mathbb{A}_F).
\end{split}
\end{equation}
\end{lemma}

\subsection{Fourier transform and Poisson summation}
We want to define the Schwartz function on the space of quadrics $X_i(\mathbb{A}_F)$ and its corresponding Fourier transform.  For $f \in \mathcal{S}(V_i(\mathbb{A}_F) \oplus \mathbb{A}_F^2)$, we define the integral transform $I(f)$ as a smooth function of $X_i^{0}(\mathbb{A}_F)$ as
 \begin{equation}
\label{eq:I}
\begin{split}
I(f) : 
\xi &\mapsto \int_{N_2(\mathbb{A}_F)\backslash \mathrm{SL}_2(\mathbb{A}_F)} r_i(g) f(\xi, 0, 1) dg .
\end{split}
\end{equation}
Then we define the space of Schwartz functions as follows
\begin{definition}
The Schwartz space of $X_i(\mathbb{A}_F)$ is defined as the image of the integral transform of \eqref{eq:I}
\begin{equation*}
\begin{split}
\mathcal{S}(X_{i}(\mathbb{A}_F)) := \mathrm{Im}\left(\mathcal{S}(V_i(\mathbb{A}_F) \oplus \mathbb{A}_F^2)  \to C^{\infty}(X_i^0(\mathbb{A}_F))   \right).
\end{split}
\end{equation*}
\end{definition}
If denote by 
\begin{equation*}
\begin{split}
\mathcal{S}'(X_{i}(\mathbb{A}_F)) := \mathrm{Im}\left(\mathcal{S}(V_i(\mathbb{A}_F))  \to C^{\infty}(X_i^0(\mathbb{A}_F))   \right)
\end{split}
\end{equation*}
the space of smooth functions obtained by the restriction of Schwartz functions of $V_i$ to $X_i^0$, then we have \cite[Lemma 4.7]{getz2023summation}
\begin{proposition}
\label{prop:schwartz}
\begin{equation*}
\begin{split}
\mathcal{S}'(X_{i}(\mathbb{A}_F))  < \mathcal{S}(X_{i}(\mathbb{A}_F)). 
\end{split}
\end{equation*}
\end{proposition}
The integral transform \eqref{eq:I}  is understood as the tensor products of $I_v$ defined at each local place $v$ of $F$. For  $f_v \in \mathcal{S}(V_i(F_v) \oplus F_v^2)$, $\xi \in X_i^{0}(F_v)$ we have
 \begin{equation*}
\begin{split}
I_v(f_v) : 
\xi & \mapsto \int_{N_2(F)\backslash \mathrm{SL}_2(F)} r_i(g) f_v(\xi, 0, 1) dg.
\end{split}
\end{equation*}
Then Proposition \ref{prop:schwartz}
implies given a function $f_v \in \mathcal{S}(V_i(F_v))$ over the local place $v$ of $F$ one can choose
$\tilde{f}_v \in  \mathcal{S}(V_i(\mathbb{A}_F) \oplus \mathbb{A}_F^2)$  such that $I_v(\tilde{f}_v)  = f_v|_{X_i^0(F_v)}$. We will give a proof of this at Archimedean place in the next section.  
Following \cite[lemma 4.5]{getz2023summation}, the regularized value of $I(f)$ at $0 \in X_i(\mathbb{A}_F)$ is defined as the special value of a Tate integral
\begin{equation*}
\begin{split}
Z_i(f, s) := \int_{\mathbb{A}^{\times}_F}   |a|^s \chi_{Q_i}(a) \int_{K} r_i(k) f(0_{V_i}, 0, a) dk  d^{\times} a.
\end{split}
\end{equation*}
This allows us to define the constant term as follows
\begin{equation*}
\begin{split}
c_i(f) := \begin{cases}
     \lim_{s \to 0} \frac{d}{ds} \left(s Z_i(f, s +2 - \frac{\mathrm{dim} V_i}{2})\right) \quad & \text{if $2 - \frac{\mathrm{dim} V_i}{2}$ is the pole of $Z_i(f, s)$}\\
     Z_i(f, 2 - \frac{\mathrm{dim} V_i}{2})    & \text{otherwise}.
\end{cases}
\end{split}
\end{equation*}

Define a symplectic Fourier transform of  $\mathcal{S}(V_i(\mathbb{A}_F) \oplus \mathbb{A}_F^2)$ on the $\mathbb{A}^2_F$ part
\begin{equation*}
\begin{split}
\mathcal{F}_{V_i}: \mathcal{S}(V_i(\mathbb{A}_F) \oplus \mathbb{A}_F^2) &\to \mathcal{S}(V_i(\mathbb{A}_F) \oplus \mathbb{A}_F^2) \\
f &\mapsto \left( (\xi, u_1, u_2) \mapsto \int_{\mathbb{A}^2_F} f(\xi, w_1, w_2) \psi(w_1 u_2 - w_2 u_1) dw_1 dw_2 \right).
\end{split}
\end{equation*}
It is easy to see $\mathcal{F}_{V_i}$ commutes with the extended Weil representation $r_i$ of $\mathrm{SL}_2(\mathbb{A}_F)$,
\begin{equation*}
\begin{split}
\mathcal{F}_{V_i} \circ r_{i}(g)  = r_i(g) \circ \mathcal{F}_{V_i}, \quad g \in \mathrm{SL}_2(\mathbb{A}_F).
\end{split}
\end{equation*}
Now if $\tilde{f}\in  \mathcal{S}(V_i(\mathbb{A}_F) \oplus \mathbb{A}_F^2)$ such that $f = I(\tilde{f}) \in \mathcal{S}(X_i(\mathbb{A}_F))$, then the Fourier transform of $f$ is defined by
\begin{equation}
\label{eq:Fx}
\begin{split}
\mathcal{F}_{X_i} : \mathcal{S}(X_i(\mathbb{A}_F)) &\to \mathcal{S}(X_i(\mathbb{A}_F)) \\
 f & \mapsto I(F_{V_i}(\tilde{f})).
\end{split}
\end{equation}

Next we define a transform
\begin{equation*}
\begin{split}
d_{i+1, i} :   \mathcal{S}(V_{i+1}(\mathbb{A}_F) \oplus \mathbb{A}_F^2) &\to \mathcal{S}(V_{i}(\mathbb{A}_F) \oplus \mathbb{A}_F^2) \\
f &\mapsto \mathcal{F}_{2, i}(f |_{V_{i+1}}).
\end{split}
\end{equation*}
For $i > i' \geq 0$, define a chain of those transforms
\begin{equation*}
\begin{split}
d_{i+1, i'} := d_{i'+1, i'} \circ \cdots \circ d_{i+1, i}: \mathcal{S}(V_{i+1}(\mathbb{A}_F) \oplus \mathbb{A}_F^2) &\to \mathcal{S}(V_{i'}(\mathbb{A}_F) \oplus \mathbb{A}_F^2).
\end{split}
\end{equation*}
By convention, let $d_{i, i}$ denote the identity operator.  It has the following property
\begin{lemma}
\label{d sigma}
Let $f \in \mathcal{S}(V_{i+1}(\mathbb{A}_F) \oplus \mathbb{A}_F^2)$. For $h \in \mathrm{GO}_{Q_{i+1}}(\mathbb{A}_F)$, we have
\begin{equation*}
\begin{split}
&d_{i+1, i}(\sigma_{i+1}(h)f) = \sigma_{i}(h) d_{i+1, i}(f). \\
\end{split}
\end{equation*}
\end{lemma}
\begin{proof}
By Lemma \ref{lemma:h} in the next section of local computations, 
\begin{equation*}
\begin{split}
(\sigma_{i+1}(h) f) |_{V_{i+1}}  = R(h) (f|_{V_{i+1}}).
\end{split}
\end{equation*}
Then the result follows from the composition of following operators applied to $f|_{V_{i+1}}$
\begin{equation*}
\begin{split}
& \mathcal{F}_{2, i} \circ R(h) =\mathcal{F}_{2, i} \circ R(h) \circ \mathcal{F}^{-1}_{2, i} \circ \mathcal{F}_{2, i} =  \sigma_i(h) \circ \mathcal{F}_{2, i}. \\
\end{split}
\end{equation*}
\end{proof}
The main result in \cite{getz2023summation} is the following Poisson summation formula
\begin{theorem}\label{thm:possion summation} 
\cite[Theorem 1.2]{getz2023summation}
Let $f \in \mathcal{S}(V_3(\mathbb{A}_F) \oplus \mathbb{A}_F^2)$.  Then
\begin{equation*}
\begin{split}
& \sum_{i = 1}^3 c_i(d_{3, i}(f)) + \sum_{i=1}^3 \sum_{\xi \in X_i^0(F)}  I(d_{3, i} (f)) (\xi) + d_{3, 0}(f) (0) \int_{[ \mathrm{SL}_2 ]} 1 dg   \\
& = \sum_{i = 1}^3 c_i(d_{3, i}(\mathcal{F}_{V_3}(f))) + \sum_{i=1}^3 \sum_{\xi \in X_i^0(F)} I(d_{3, i} (\mathcal{F}_{V_3}( f))) (\xi) + d_{3, 0}(\mathcal{F}_{V_3}(f)) (0) \int_{[ \mathrm{SL}_2 ]} 1 dg.
\end{split}
\end{equation*}
\end{theorem}
For $h \in B$, we write $\sigma_3(h)$ for $\sigma_3(\iota_3(h))$.  In the rest of the section we will state several properties when 
$f$ is replaced by $\sigma_3(h) f$ in the above Poisson summation formula.
\begin{lemma}
\label{lemma:V3}
Let $f \in \mathcal{S}(V_3(\mathbb{A}_F) \oplus \mathbb{A}_F^2)$. For $h  \in B$ and $\xi \in X^0_3$, we have
\begin{equation*}
\begin{split}
I(\sigma_3(h) f)(\xi) &= |\lambda(h)|^{-1} I(f)(\iota(h) \cdot \xi), \\
I(\mathcal{F}_{V_3}  (\sigma_3(h) f))(\xi) &= |\lambda(h)|^{-3} I(\mathcal{F}_{V_3}(f))(\lambda(h)^{-1} \iota(h) \cdot \xi).
\end{split}
\end{equation*}
\end{lemma}
\begin{proof}
This follows from Proposition \ref{prop:h}.
\end{proof}
\begin{lemma}
\label{lemma:V2}
Let $f \in \mathcal{S}(V_2(\mathbb{A}_F) \oplus \mathbb{A}_F^2)$.  Let $h  \in B$ be expressed as in the form \eqref{eq:h form}. 
Write $\sigma_2(h)$ for $\sigma_2(\iota_2(h))$. For $\xi = (x_1, x_2, x_3, x_4) \in X^0_2$,
\begin{equation*}
\begin{split}
I\left(\sigma_2\left(   \begin{matrix}
    a_t & \\
        &1
\end{matrix}  \right) f\right)(\xi) &=  I(f)\left( \left( \begin{matrix}
    I_2 & \\
        &t^2 & \\
        & & t^{-2}
\end{matrix}   \right) \xi \right), \\
I\left(\sigma_2\left(   \begin{matrix}
    n_u & \\
        &1
\end{matrix}  \right) f\right)(\xi) &=  I(f) \left( \left( \begin{matrix}
   1 &0 &u &0\\
   0 &1 &-u &0\\
   0 &0 &1  &0\\
   u &-u &u^2 &1
\end{matrix}   \right) \xi \right),\\
I \left( \sigma_2 \left(  \begin{matrix}
    d_{\mu} & \\
        &d_{\mu}
\end{matrix}  \right) f \right) (\xi) &=  |\lambda(h)|^{-3/2} I(f)(\xi), \\
I\left( \sigma_2\left(  \begin{matrix}
    1 & \\
        &a_t
\end{matrix}  \right) f \right)(\xi) &=  |t|^{2} I(f)(t^2 \xi), \\
I\left( \sigma_2 \left( \begin{matrix}
    1 & \\
        &n_v
\end{matrix}  \right)  f\right)(\xi) &=  \psi(v x_1 + v x_2) I(f)(\xi).
\end{split}
\end{equation*}
\end{lemma}
\begin{proof}
The first two equations follow from Proposition \ref{prop:h}  applied to 
\begin{equation*}
\begin{split}
\iota_2 \left( h   \right) = \left( \begin{matrix}
    h & \\
        &1 & \\
        & &1
\end{matrix}  \right)
\end{split}
\end{equation*}
if $h =  \left(   \begin{matrix}
    a_t & \\
        &1
\end{matrix}  \right)$ or $h = \left(   \begin{matrix}
    n_u & \\
        &1
\end{matrix}  \right)$. 
The third equation follows from Proposition \ref{prop:t} applied to
\begin{equation*}
\begin{split}
\iota_2 \left( \begin{matrix}
    d_{\mu} & \\
        &d_{\mu}
\end{matrix}   \right) =  \mu^2 I_{V_3}.
\end{split}
\end{equation*}
The forth equation follows from Proposition \ref{prop:a} applied to
\begin{equation*}
\begin{split}
\iota_2 \left( \begin{matrix}
    1 & \\
        &a_{t}
\end{matrix}   \right) = \left( \begin{matrix}
    I_{V_2} & \\
        &t^2 & \\
        & &t^{-2}
\end{matrix}  \right).
\end{split}
\end{equation*}
The last one follows from Proposition \ref{prop:u}.
\end{proof}
\begin{lemma}
\label{lemma:V1}
Let $f \in \mathcal{S}(V_1(\mathbb{A}_F) \oplus \mathbb{A}_F^2)$. 
For $h  \in B'$, write $\sigma_1(h)$ for $\sigma_1(\iota_1(h))$. 
For $\xi = (x_1, x_2) \in X^0_1$, we have
\begin{equation*}
\begin{split}
I\left(\sigma_1\left(   \begin{matrix}
    a_t & \\
        &1
\end{matrix}  \right) f\right)(\xi) &=  I(f)\left(t^2 \xi \right), \\
I\left(\sigma_1\left(   \begin{matrix}
    n_u & \\
        &1
\end{matrix}  \right) f\right)(\xi) &=  \psi(-u x_1 + u x_2) I(f)(\xi), \\
I \left( \sigma_1 \left(  \begin{matrix}
    d_{\mu} & \\
        &d_{\mu}
\end{matrix}  \right) f \right) (\xi) &=  |\lambda(h)|^{-1} I(f)(\xi).
\end{split}
\end{equation*}
\end{lemma}
\begin{proof}
The first equation follows from Proposition \ref{prop:a} applied to 
\begin{equation*}
\begin{split}
\iota_1 \left( \begin{matrix}
    a_t & \\
        &1
\end{matrix}   \right) = \left( \begin{matrix}
    I_{V_1} & \\
        &t^2 & \\
        & &t^{-2}
\end{matrix}  \right).
\end{split}
\end{equation*}
The second equation follows from Proposition \ref{prop:u}.
The third equation follows from Proposition \ref{prop:t} applied to
\begin{equation*}
\begin{split}
\iota_1 \left( \begin{matrix}
    d_{\mu} & \\
        &d_{\mu}
\end{matrix}   \right) =  \mu^2 I_{V_2}.
\end{split}
\end{equation*}
\end{proof}
Let $f \in \mathcal{S}(V_3(\mathbb{A}_F) \oplus \mathbb{A}_F^2)$. For $h \in B(\mathbb{A}_F)$ and $\gamma \in B(F)$, it is easy to see from Lemma \ref{lemma:V3} that
\begin{equation*}
\begin{split}
 \sum_{\xi \in X_3^0(F)}  I (  \sigma_3(h)  f) (\xi) = \sum_{\xi \in X_3^0(F)}  I(  \sigma_3(\gamma) \sigma_3(h )  f) (\xi) =
\sum_{\xi \in X_3^0(F)}  I(  \sigma_3(h \gamma )  f) (\xi).
\end{split} 
\end{equation*}
The Lemma \ref{d sigma} gives
\begin{equation*}
\begin{split}
I (d_{3,2} (\sigma_3(h) f))  = I(\sigma_2(h)d_{3,2}(f)).
\end{split}
\end{equation*}
Then from Lemma \ref{lemma:V2},  we get
\begin{equation*}
\begin{split}
\sum_{\xi \in X_2^0(F)}  I(d_{3, 2} (  \sigma_3(h)  f)) (\xi) = 
\sum_{\xi \in X_2^0(F)}  I(d_{3, 2} (  \sigma_3(h \gamma )  f)) (\xi).
\end{split} 
\end{equation*}
Similarly for $h \in B'(\mathbb{A}_F)$ and $\gamma \in B'(F)$,  applying Lemma \ref{d sigma} gives
\begin{equation*}
\begin{split}
I (d_{2,1} d_{3,2} (\sigma_3(h) f))  = I( d_{2,1}\sigma_2(h)d_{3,2}(f)) = I( \sigma_1(h) d_{2,1}d_{3,2}(f)).
\end{split}
\end{equation*}
From Lemma \ref{lemma:V1}, we have
\begin{equation*}
\begin{split}
\sum_{\xi \in X_1^0(F)}  I(d_{3, 1} (  \sigma_3(h)  f) )(\xi)  = 
\sum_{\xi \in X_1^0(F)}  I(d_{3, 1} (  \sigma_3(h\gamma)  f)) (\xi).
\end{split} 
\end{equation*}
Applying Lemma \ref{lemma:u} to  $\gamma = \left( \begin{matrix}
    1 & \\
        &n_v
\end{matrix}   \right) \in B(F)$ and Lemma \ref{lemma:a} to $\gamma = \left( \begin{matrix}
    1 & \\
        &a_{t}
\end{matrix}   \right) \in B(F)$ gives
\begin{equation*}
\begin{split}
I (d_{2,1} d_{3,2} (\sigma_3(h \gamma) f))  = I( d_{2,1}\sigma_2(\gamma)d_{3,2}(\sigma_3(h)f)) = I( d_{2,1}d_{3,2}( \sigma_3(h)f)).
\end{split}
\end{equation*}
The same arguments apply to $\mathcal{F}_{V_3}(f)$ by first applying Lemma \ref{lemma:F h}. Therefore by combining all results from the section of local computations we have shown
\begin{proposition}
Let $f \in \mathcal{S}(V_3(\mathbb{A}_F) \oplus \mathbb{A}_F^2)$. For $h \in B(\mathbb{A}_F)$ and $\gamma \in B(F)$,
\begin{equation*}
\begin{split}
\sum_{\xi \in X_i^0(F)}  I(d_{3, i} (  \sigma_3(h)  f)) (\xi) &=
\sum_{\xi \in X_i^0(F)}  I(d_{3, i} (  \sigma_3(h \gamma)  f)) (\xi), \\ \sum_{\xi \in X_i^0(F)}  I(d_{3, i} ( \mathcal{F}_{V_3} (\sigma_3(h)  f))) (\xi) &=
\sum_{\xi \in X_i^0(F)} I(d_{3, i} ( \mathcal{F}_{V_3} (\sigma_3(h \gamma )  f))) (\xi).
\end{split}
\end{equation*}
\end{proposition}
Moreover, by Proposition \ref{prop:u} we have vanishing integration over the unipotent subgroup of $B$.
\begin{proposition}
\label{prop:vanishing nilpotent}
Let $f \in \mathcal{S}(V_3(\mathbb{A}_F) \oplus \mathbb{A}_F^2)$. For $h \in B(\mathbb{A}_F)$ and  $\xi \in X_i^0(F)$, $i = 1, 2$,  we have
\begin{equation*}
\begin{split}
 \int_{N(\mathbb{A}_F) / N(F)} I(d_{3, i} (  \sigma_3(h n )  f)) (\xi) dn &= 0, \\
\int_{N(\mathbb{A}_F) / N(F)}  I(d_{3, i} ( \mathcal{F}_{V_3} (\sigma_3(hn )  f))) (\xi) dn &= 0.
\end{split}
\end{equation*}
\end{proposition}
In the next section we are going to define and study the zeta integrals of group $B$ at the Archimedean place. For this we choose a test function
\begin{equation*}
\begin{split}
f = f_{\infty} 1_{V_4({\hat{\mathbb{Z}}})} \in \mathcal{S}(V_3(\mathbb{A}_\mathbb{Q}) \oplus \mathbb{A}_\mathbb{Q}^2)
\end{split}
\end{equation*}
Then by \cite[(2.1)]{getz2023summation} we have
\begin{proposition}
\begin{equation*}
\label{prop:summation at Archimedean place}
\begin{split}
\sum_{\xi \in X_i^0(\mathbb{Q})}  I(d_{3, i} (  f )) (\xi) &=
\sum_{n=1}^{\infty} \sum_{\xi \in n X_i^0(\mathbb{Z})} n^{\frac{\mathrm{dim \ V_i}}{2} - 2} I(d_{3, i} (f_{\infty})) (\xi), \\ 
\sum_{\xi \in X_i^0(\mathbb{Q})}  I(d_{3, i} (  \mathcal{F}_{V_3} ( f ))) (\xi) &=
\sum_{n=1}^{\infty} \sum_{\xi \in n X_i^0(\mathbb{Z})} n^{\frac{\mathrm{dim \ V_i}}{2} - 2} I(d_{3, i} (  \mathcal{F}_{V_3} (f_{\infty}))) (\xi).
\end{split}
\end{equation*}    
\end{proposition}
\subsection{Local computations}
Let $v$ be a place of $F$ and denote by $F = F_v$ the local field in this section.
\begin{lemma}
\label{lemma:h}
Let  $f \in \mathcal{S}(V_i(F) \oplus F^2)$. For  $h \in \mathrm{GO}_{Q_{i}}(F)$ and $(\xi, \xi_1, \xi_2) \in V_i(F) \oplus F^2$, we have
\begin{equation*}
\begin{split}
\sigma_{i}(h) f(\xi, \xi_1, \xi_2)  = f(h \cdot \xi, \lambda(h) \xi_1, \xi_2).
\end{split}
\end{equation*}
In particular, 
\begin{equation*}
\begin{split}
(\sigma_{i}(h) f) |_{V_i}  = R(h) (f|_{V_i}). 
\end{split}
\end{equation*}

\end{lemma}
\begin{proof}
Calculate
\begin{equation*}
\begin{split}
R(h)  \mathcal{F}_{2, i}^{-1} ( f) ( \xi, \xi_1, \xi_2) & = \int_{F} f(h \cdot \xi, \lambda(h) \xi_1, v) \psi(-\xi_2 v) dv.
\end{split}
\end{equation*}
Applying $\mathcal{F}_{2, i}$ we get
\begin{equation*}
\begin{split}
\sigma_{i}(h) f(\xi, \xi_1, \xi_2) & = \int_{F} \int_{F} f(h \cdot \xi, \lambda(h) \xi_1, v) \psi(- u v) \psi(\xi_2, u) dv du \\
&  = f(h \cdot \xi, \lambda(h) \xi_1, \xi_2).
\end{split}
\end{equation*}
In particular it implies
\begin{equation*}
\begin{split}
\sigma_{i}(h) f(\xi, 0, 0)
&  = f(h \cdot \xi, 0, 0).
\end{split}
\end{equation*}
\end{proof}
\begin{lemma}
\label{lemma:u}
Let $f \in \mathcal{S}(V_i(F) \oplus F^2)$. For  $u(x)$ given by \eqref{eq:u form} and $(\xi, \xi_1, \xi_2) \in V_i(F) \oplus F^2$, we have 
\begin{equation*}
\begin{split}
\sigma_i(u(x))f(\xi, \xi_1, \xi_2) = f(\xi + (J_i x)^t \xi_1, \xi_1, \xi_2 ) \psi(\xi_2 ( x^t \xi +  Q_i(x) \xi_1  )).
\end{split}
\end{equation*}
In particular, 
\begin{equation*}
\begin{split}
(\sigma_{i}(u(x)) f) |_{V_i}  = f|_{V_i}. 
\end{split}
\end{equation*}
\end{lemma}
\begin{proof}
Calculate
\begin{equation*}
\begin{split}
&\sigma_i(u(x))f(\xi, \xi_1, \xi_2) \\
&=\int_F \left(\int_F f(\xi + (J_i x)^t \xi_1, \xi_1, v ) \bar{\psi}((- x^t \xi -  Q_i(x) \xi_1 +u ) v) dv \right)  \psi(\xi_2u) du.
\end{split}
\end{equation*}
Changing variable $u \to u - (- x^t \xi -  Q_i(x) \xi_1)$ and making use of Fourier inversion we get
\begin{equation*}
\begin{split}
&\int_F \left(\int_F f(\xi + (J_i x)^t \xi_1, \xi_1, v ) \bar{\psi}(u v) dv \right)  \psi(\xi_2 u) \psi(\xi_2  (x^t \xi + Q_i(x) \xi_1      )    )   du \\
& =  f(\xi + (J_i x)^t \xi_1, \xi_1, \xi_2 ) \psi(\xi_2  (x^t \xi + Q_i(x) \xi_1      )    )   du. 
\end{split}
\end{equation*}
\end{proof}

\begin{lemma}
\label{lemma:a}
Let  $f \in \mathcal{S}(V_i(F) \oplus F^2)$. For  $(\xi, \xi_1, \xi_2) \in V_i(F) \oplus F^2$, we have 
\begin{equation*}
\begin{split}
\sigma_i \left( \begin{matrix}
    I_{V_{i}} & \\
      &t  & \\
        & &t^{-1}
\end{matrix}  \right) f  (\xi, \xi_1, \xi_2) &= |t| f(\xi, t \xi_1, t\xi_2).
\end{split}
\end{equation*}
\end{lemma}
\begin{proof}
Calculate
\begin{equation*}
\begin{split}
\sigma_i \left( \begin{matrix}
    I_{V_{i}} & \\
      &t  & \\
        & &t^{-1}
\end{matrix}  \right) f  (\xi, \xi_1, \xi_2) &= \int_F \left( \int_F f(\xi,a \xi_1, v)  \bar{\psi}(t^{-1} u v) dv  \right) \psi(u \xi_2) du.
\end{split}
\end{equation*}
Changing variable $u \to t u$ and using Fourier inversion we get $|t| f(\xi, t \xi_1, t\xi_2)$.
\end{proof}
\begin{lemma}
\label{lemma:center}
Let  $f \in \mathcal{S}(V_i(F) \oplus F^2)$. We have 
\begin{equation*}
\begin{split}
\sigma_i \left( t I_{V_{i+1}}  \right) f  &=  \chi_{Q_i}(t)|t|^{-\frac{\mathrm{dim} V_i}{2} - 1} r_i \left( \begin{matrix}
    t& \\
      & t^{-1}
\end{matrix}  \right) f.
\end{split}
\end{equation*}
\end{lemma}
\begin{proof}
By Lemma \ref{lemma:a} and \ref{lemma:h} we have
\begin{equation}
\label{eq:center formula}
\begin{split}
\sigma_i \left( \begin{matrix}
    I_{V_{i}} & \\
      &t  & \\
        & &t^{-1}
\end{matrix}  \right) f &=  \chi_{Q_i}(t)|t|^{\frac{\mathrm{dim} V_i}{2} + 1} r_i \left( \begin{matrix}
    t^{-1}& \\
      & t
\end{matrix}  \right)  \sigma_i \left( \begin{matrix}
    t I_{V_{i}} & \\
      &t^2  & \\
        & &1
\end{matrix}  \right) f.
\end{split}
\end{equation}
Since
\begin{equation*}
\begin{split}
t I_{V_{i+1}} = \left( \begin{matrix}
     I_{V_{i}} & \\
      &t^{-1}  & \\
        & &t
\end{matrix}  \right)  \left( \begin{matrix}
    t I_{V_{i}} & \\
      &t^2  & \\
        & &1
\end{matrix}  \right),
\end{split}
\end{equation*}
now we apply \eqref{eq:center formula} to get
\begin{equation*}
\begin{split}
\sigma_i \left( t I_{V_{i+1}}  \right) f  &=  \chi_{Q_i}(t)|t|^{-\frac{\mathrm{dim} V_i}{2} - 1} r_i \left( \begin{matrix}
    t& \\
      & t^{-1}
\end{matrix}  \right) f.
\end{split}
\end{equation*}
\end{proof}

\begin{lemma}
\label{lemma:F h}
Let  $f \in \mathcal{S}(V_i(F) \oplus F^2)$. For  $h \in \mathrm{GO}_{Q_{i}}(F)$, we have  
\begin{equation*}
\begin{split}
\mathcal{F}_{V_i}  (\sigma_i( h) f)  &= |\lambda(h)|^{-\frac{\mathrm{dim} V_i}{2} - 1}
r_i \left( \begin{matrix}
      \lambda(h) &  \\
        & \lambda(h)^{-1}
\end{matrix}  \right)  \sigma_i(\lambda(h)^{-1} h)  \mathcal{F}_{V_i} (f),
\end{split}
\end{equation*}
and 
\begin{equation*}
\begin{split}
\sigma_i( h) \mathcal{F}_{V_i}  ( f)  &= |\lambda(h)|^{-\frac{\mathrm{dim} V_i}{2} - 1}
r_i \left( \begin{matrix}
      \lambda(h) &  \\
        & \lambda(h)^{-1}
\end{matrix}  \right)    \mathcal{F}_{V_i} ( \sigma_i(\lambda(h)^{-1} h)f).
\end{split}
\end{equation*}
In particular, if $\lambda(h) = 1$, 
\begin{equation*}
\begin{split}
\mathcal{F}_{V_i}  (\sigma_i(h) f) &= \sigma_i(h) \mathcal{F}_{V_i} (f).
\end{split}
\end{equation*}
\end{lemma}
\begin{proof}
The results follows from the factorization
\begin{equation*}
\begin{split}
\left( \begin{matrix}
    h & \\
      &\lambda(h)  & \\
        & &1
\end{matrix}  \right) \left( \begin{matrix}
    I_{V_{3}} & \\
      &  &1 \\
        &1 &
\end{matrix}  \right) = \left( \begin{matrix}
    I_{V_{i}} & \\
      &  &1 \\
        &1 &
\end{matrix}  \right) \left( \begin{matrix}
    \lambda(h)^{-1} h & \\
      &\lambda(h)^{-1}  & \\
        & &1
\end{matrix}  \right) \lambda(h) I_{V_{i+1}} , 
\end{split}
\end{equation*}
and from \cite[Proposition 4.3 (4.19)]{getz2023summation}
\begin{equation*}
\begin{split}
\sigma_i \left( \begin{matrix}
    I_{V_{i}} & \\
      &  &1 \\
        &1 &
\end{matrix}  \right) = \mathcal{F}_{V_i}. 
\end{split}
\end{equation*}
Since $ \chi_{Q_i}$ is trivial on the similitude norm of  $h$ by \cite[Lemma 3.2]{getz2019summation}, now we apply Lemma \ref{lemma:center} to get the first equation. The second one can be proved similarly.
\end{proof}
For  $h \in \mathrm{GO}_{Q_{i}}(F)$, because 
\begin{equation*}
\begin{split}
\left( \begin{matrix}
    h & \\
      &1  & \\
        & &\lambda(h)
\end{matrix}  \right)  = \left( \begin{matrix}
    \lambda(h)^{-1} h & \\
      &\lambda(h)^{-1}  & \\
        & &1
\end{matrix}  \right) \lambda(h) I_{V_{i+1}} , 
\end{split}
\end{equation*}
we have proved
\begin{lemma}
\label{lemma: sigma(h')}
\begin{equation*}
\begin{split}
\sigma_i \left( \begin{matrix}
    h & \\
      &1  & \\
        & &\lambda(h)
\end{matrix}  \right)  = |\lambda(h)|^{-\frac{\mathrm{dim} V_i}{2} - 1}
r_i \left( \begin{matrix}
      \lambda(h) &  \\
        & \lambda(h)^{-1}
\end{matrix}  \right)  \sigma_i(\lambda(h)^{-1} h).
\end{split}
\end{equation*}
\end{lemma}

\begin{proposition}
\label{prop:h}
Let  $f \in \mathcal{S}(V_i(F) \oplus F^2)$. For  $h \in \mathrm{GO}_{Q_{i}}(F)$ and $\xi \in X^0_i(F)$, we have
\begin{equation*}
\begin{split}
I(\sigma_i(h) f)(\xi) &= |\lambda(h)|^{-1} I(f)(h \cdot \xi), \\
I(\mathcal{F}_{V_i} ( \sigma_i(h) f))(\xi) &= |\lambda(h)|^{-\mathrm{dim} \ V_i / 2} I(\mathcal{F}_{V_i}(f))(\lambda(h)^{-1} h \cdot \xi).
\end{split}
\end{equation*}
\end{proposition}
\begin{proof}
See \cite[Proposition 4.3 (4.16) and Corollary 4.4]{getz2023summation}. Note that the action $\mathrm{GO}_{Q_{i}}(F)$ is defined as $R(h)f(\xi) = f(h \cdot \xi)$ instead of the left action $L(h)f(\xi) = f(h^{-1} \cdot \xi)$ there.
\end{proof}

\begin{proposition}
\label{prop:u}
Let $f \in \mathcal{S}(V_i(F) \oplus F^2)$. For  $u(x)$ given by \eqref{eq:u form} and $\xi \in X^0_i(F)$, we have 
\begin{equation*}
\begin{split}
I(\sigma_i(u(x))f) (\xi) = \psi(x^t \xi) I(f)(\xi).
\end{split}
\end{equation*}
\end{proposition}
\begin{proof}
This follows from applying definition of integral transform $I$ in \eqref{eq:I} to Lemma \ref{lemma:u}.
\end{proof}

\begin{proposition} \label{prop:a}
Let  $f \in \mathcal{S}(V_i(F) \oplus F^2)$.   For $\xi \in X^0_i(F)$,
\begin{equation*}
\begin{split}
I\left(\sigma_i \left( \begin{matrix}
    I_{V_{i}} & \\
      &t  & \\
        & &t^{-1}
\end{matrix}  \right) f \right) (\xi) &= \chi_{Q_i}(t) |t|^{\frac{\mathrm{dim} \ V_i }{2} - 1} I(f)(t \xi), \\
I\left(\mathcal{F}_{V_i} \left(  \sigma_i \left( \begin{matrix}
    I_{V_{i}} & \\
      &t  & \\
        & &t^{-1}
\end{matrix}  \right) f \right) \right) (\xi) &= \chi_{Q_i}(t) |t|^{1 -  \frac{\mathrm{dim} \ V_i }{2}}  I(\mathcal{F}_{V_i} (f) )(t^{-1} \xi).
\end{split}
\end{equation*}
\end{proposition}
\begin{proof}
The first equation follows from \eqref{eq:center formula} and Proposition \ref{prop:h}.
From factorization
\begin{equation*}
\begin{split}
\left( \begin{matrix}
    I_{V_{i}} & \\
      &  &1 \\
        &1 &
\end{matrix}  \right) \left( \begin{matrix}
    I_{V_{i}} & \\
      &t  & \\
        & &t^{-1}
\end{matrix}  \right) = \left( \begin{matrix}
    t^{-1} I_{V_{i}} & \\
      &t^{-2}  & \\
        & &1
\end{matrix}  \right) t I_{V_{i+1}} \left( \begin{matrix}
    I_{V_{i}} & \\
      &  &1 \\
        &1 &
\end{matrix}  \right)
\end{split}
\end{equation*}
and from Proposition \ref{prop:h}, we have
\begin{equation*}
\begin{split}
I\left(\mathcal{F}_{V_i} \left(  \sigma_i \left( \begin{matrix}
    I_{V_{i}} & \\
      &t  & \\
        & &t^{-1}
\end{matrix}  \right) f \right) \right) (\xi) &= |t|^2  I(\sigma_i( t I_{V_{i+1}})  \mathcal{F}_{V_i} (f) )(t^{-1} \xi).
\end{split}
\end{equation*}
Since
\begin{equation*}
\begin{split}
t I_{V_{i+1}} = \left( \begin{matrix}
     I_{V_{i}} & \\
      &t^{-1}  & \\
        & &t
\end{matrix}  \right)  \left( \begin{matrix}
    t I_{V_{i}} & \\
      &t^2  & \\
        & &1
\end{matrix}  \right),
\end{split}
\end{equation*}
By first equation and Proposition \ref{prop:h} again, we get
\begin{equation*}
\begin{split}
I(\sigma_i( t I_{V_{i+1}})f )( \xi) &= \chi_{Q_i}(t) |t|^{1 - \frac{\mathrm{dim} \ V_i }{2}}  I\left( \sigma_i \left( \begin{matrix}
    tI_{V_{i}} & \\
      &t^2  & \\
        & &1
\end{matrix}  \right)f \right)(t^{-1} \xi) \\
& = \chi_{Q_i}(t) |t|^{-1 -  \frac{\mathrm{dim} \ V_i }{2}}  I(f )(\xi).
\end{split}
\end{equation*}
Therefore
\begin{equation*}
\begin{split}
|t|^2  I(\sigma_i( t I_{V_{i+1}})  \mathcal{F}_{V_i} (f) )(t^{-1} \xi)
& = \chi_{Q_i}(t) |t|^{1 -  \frac{\mathrm{dim} \ V_i }{2}}  I(\mathcal{F}_{V_i} (f) )(t^{-1} \xi).
\end{split}
\end{equation*}
\end{proof}

\begin{proposition} \label{prop:t}
Let  $f \in \mathcal{S}(V_i(F) \oplus F^2)$. For   $\xi \in X^0_i(F)$,
\begin{equation*}
\begin{split}
I\left(\sigma_i \left( t I_{V_{i+1}} \right) f \right) (\xi) &= \chi_{Q_i}(t) |t|^{-1 - \frac{\mathrm{dim} \ V_i }{2}} I(f)(\xi), \\
I\left(\mathcal{F}_{V_i} \left(  \sigma_i \left(t I_{V_{i+1}}  \right) f \right) \right) (\xi) &= \chi_{Q_i}(t) |t|^{-1 - \frac{\mathrm{dim} \ V_i }{2}}  I(\mathcal{F}_{V_i} (f) )(\xi).
\end{split}
\end{equation*}
\end{proposition}
\begin{proof}
The first equation has been proved in the second part of the Proposition \ref{prop:a}. The second equation follows from
\begin{equation*}
\begin{split}
\left( \begin{matrix}
    I_{V_{i}} & \\
      &  &1 \\
        &1 &
\end{matrix}  \right) \left( \begin{matrix}
    t I_{V_{i}} & \\
      &t  & \\
        & &t
\end{matrix}  \right) = \left( \begin{matrix}
    t I_{V_{i}} & \\
      &t  & \\
        & &t
\end{matrix}  \right) \left( \begin{matrix}
    I_{V_{i}} & \\
      &  &1 \\
        &1 &
\end{matrix}  \right),
\end{split}
\end{equation*}
which implies
\begin{equation*}
\begin{split}
I\left(\mathcal{F}_{V_i} \left(  \sigma_i \left(t I_{V_{i+1}}  \right) f \right) \right) (\xi)  &= I\left( \sigma_i \left(t I_{V_{i+1}} \right)
 \mathcal{F}_{V_i}(f) \right) (\xi)   \\
&= \chi_{Q_i}(t) |t|^{-1 - \frac{\mathrm{dim} \ V_i }{2}}  I(\mathcal{F}_{V_i} (f) )(\xi).
\end{split}
\end{equation*}
\end{proof}

\section{Zeta functions and multiple Dirichlet series}
\subsection{Zeta functions associated with quadrics}
Let $B^{+}$ be the subgroup of $B(\mathbb{R})$ consisting of positive diagonal entries. Let $V = V(\mathbb{R})$ be the real space of pairs of binary quadratic forms, and $X = X(\mathbb{R}) = \{(x, y) \in V: \mathrm{Disc}(x) = \mathrm{Disc}(y) \}$.
Let $L$ be the integral points of $X^{0} = X - \{ 0 \}$.
There are three relative invariants $P_1, P_2$ and $P$ on $X$ given by
\begin{equation*}
\begin{split}
P_1(v) := x_1, \quad P_2(v) := y_1, \quad P(v) := \mathrm{Disc}(x) = 4(x_{2}^2 - x_1 x_3) 
\end{split}
\end{equation*}
for $v = \left( \left( \begin{matrix}
    x_1    & x_{2}\\
    x_{2} & x_3
\end{matrix} \right), \left( \begin{matrix}
    y_1    & y_{2}\\
    y_{2} & y_3
\end{matrix} \right) \right)$. 
Let $\Lambda^{+} = \Lambda \cap B^{+}$ and $A^{+}= A \cap B^{+}$.  For $h = \left( \begin{matrix}
    d_{\mu} & \\
        &d_{\mu}
\end{matrix}  \right) \left( \begin{matrix}
    a_t & \\
        &a_{s}
\end{matrix}  \right) \left( \begin{matrix}
    n_u & \\
        &n_v
\end{matrix}  \right)
\in B^{+} = \Lambda^{+} A^{+} N$ uniquely expressed in the Iwasawa decomposition, set
\begin{equation*}
\begin{split}
\chi_1(h) := \mu^2 t^2, \quad \chi_2(h) := \mu^2 s^2, \quad \chi(h) := \mu^4.
\end{split}
\end{equation*}
We write $h \cdot v$ for the action $\iota(h) \cdot v$  and write $t \cdot h$ for the product $t I_V \cdot \iota(h)$ where $\iota: B^{+} \to \mathrm{GO}_{Q}(\mathbb{R})$ denotes the embedding to the orthogonal similitude group. Then 
\begin{equation*}
\begin{split}
P_1(h \cdot v) = \chi_1(h) P_1(v), P_2(h \cdot v) = \chi_2(h) P_2(v), P(h \cdot v) = \chi(h) P(v).
\end{split}
\end{equation*}
Note that the similitude norm  $\lambda(h)$ is equal to $\chi(h)$.
The Haar measure on $B^{+}$ is defined as
\begin{equation*}
\begin{split}
d \left( 
\left( \begin{matrix}
    d_{\mu} & \\
        &d_{\mu}
\end{matrix}  \right) \left( \begin{matrix}
    a_t & \\
        &a_{s}
\end{matrix}  \right) \left( \begin{matrix}
    n_u & \\
        &n_v
\end{matrix}  \right) \right) = du dv t^{-2} s^{-2} d^{\times}t d^{\times}s  d^{\times}\mu.
\end{split}
\end{equation*}
The set of singular points is the union 
\begin{equation*}
\begin{split}
\Sigma_{\mathrm{sig}} = S_1 \cup S_2 \cup S
\end{split}
\end{equation*}
where $S_i = \{v \in X^0: P_i(v) = 0\}$ $(i = 1,2)$ and $S = \{v \in X^0: P(v) = 0\}$.  
There are two open subsets 
\begin{equation*}
\begin{split}
X^{0}_{\pm} = \{(x, y) \in X^{0}: \pm \mathrm{Disc}(x) > 0\}.
\end{split}
\end{equation*}
Each open subset has four $B^{+}$-orbits. 
The base points of four open orbits in $X^0_{+}$ are
$\left( \left( \begin{matrix}
    \pm 1    & 0\\
    0 & \mp 1
\end{matrix} \right), \left( \begin{matrix}
    \pm 1    & 0\\
    0 &  \mp 1
\end{matrix} \right) \right)$ and base points of four open orbits in $X_{-}^{0}$ are $\left( \left( \begin{matrix}
    \pm 1    & 0\\
    0 & \pm 1
\end{matrix} \right), \left( \begin{matrix}
    \pm 1    & 0\\
    0 &  \pm 1
\end{matrix} \right) \right)$.  They all have trivial stabilizers in $B^{+}$.

We next give a proof of Proposition \ref{prop:schwartz} in the Archimedean case. 
\begin{proposition}
\label{prop:Dixmier-Malliavin}
For $f' \in \mathcal{S}(V)$, there exits $\tilde{f} \in \mathcal{S}(V \oplus \mathbb{R}^2)$ such that
\begin{equation*}
\begin{split}
f :=  I(\tilde{f}) = f'|_{X^0}.
\end{split}
\end{equation*}
\end{proposition}
\begin{proof}
For $f'\in \mathcal{S}(V)$, by a theorem of Dixmier-Malliavin, we can choose $\phi_i \otimes f_i \in \mathrm{C}_c^{\infty}(\mathrm{SL}_2(\mathbb{R})) \otimes \mathcal{S}(V)$ such that for $v \in V$, 
\begin{equation*}
\begin{split}
 f'(v) = \int_{ \mathrm{SL_2}(\mathbb{R})} \sum_{i=1}^N 
\phi_i( g)  \rho(g)f_i(v) dg. \\
\end{split}
\end{equation*}
Its evaluation at $\xi \in X^0$ becomes
\begin{equation*}
\begin{split}
f'(\xi) =  \int_{N_2(\mathbb{R}) \backslash \mathrm{SL_2}(\mathbb{R})} \sum_{i=1}^N \left(\int_{N_2(\mathbb{R})} 
\phi_i(n g) dn \right) \rho(g)f_i(\xi) d\dot{g}. \\
\end{split}
\end{equation*}
If $\tilde{f}  \in \mathcal{S}(V \oplus \mathbb{R}^2)$ is defined by
\begin{equation*}
\begin{split}
\tilde{f}(v, (0, 1)g) =  \sum_{i=1}^N \left(\int_{N_2(\mathbb{R})} 
\phi_i(n g) dn \right) f_i(v) \\
\end{split}
\end{equation*}
then it is easy to see that
\begin{equation*}
\begin{split}
f'(\xi)=  I(\tilde{f})(\xi).
\end{split}
\end{equation*}
\end{proof}
Now if $\tilde{f} \in \mathcal{S}(V \oplus \mathbb{R}^2)$ such that $f = I(\tilde{f}) \in \mathcal{S}(X)$, in view of the definition \eqref{eq:Fx}, denote the Fourier transform of $f$ by
\begin{equation*}
\begin{split}
\hat{f} := I(\mathcal{F}_{V} (\tilde{f})).
\end{split}
\end{equation*}
For the rest of the paper we want to choose a special set of $\mathcal{S}(X)$ as test functions. For this we consider a larger group action $B_2^{+} \times B_2^{+}$ on $V$. The representation $(B_2^{+} \times B_2^{+}, V)$ is a prehomogeneous vector space. Let $G^{+}$ denote $B_2^{+}\times B_2^{+}$ and $G^1$  denote $G^{+} \cap ( \mathrm{SL}_2(\mathbb{R}) \times \mathrm{SL}_2(\mathbb{R}) )$. For a base point $x_{\ast}$, take any function from $\mathcal{S}(G^1)$ and denote it  by $f_{G_1, \ast}$ to indicate one could choose different function for different base point, and extend it to $G^{+}$ independent of the determinant. Let $f_1, f_2 \in C^{\infty}_c(\mathbb{R^{+}})$. Define a Schwartz function $f'$ of $V$ supported on the open orbit $G^{+} x_{\ast}$ as follows
\begin{equation} \label{eq:Schwartz}
\begin{split}
f'(hx_{\ast}) := f_{G^1,\ast}(h) f_1(\mathrm{det}(h_1)^2) f_2(\mathrm{det}(h_2)^2).
\end{split}
\end{equation}
Adding those functions associated with each base point defines a Schwartz function $f' \in \mathcal{S}(V)$. Take $f'' \in \mathcal{S}(\mathbb{R}^2)$. Define $\tilde{f} = f' \otimes f'' \in \mathcal{S}(V \oplus \mathbb{R}^2)$.
As a direct consequence of the above choice of test functions we have the following
\begin{lemma}
For $\tilde{f} = f' \otimes f'' \in \mathcal{S}(V \oplus \mathbb{R}^2)$  chosen above, both  $f$ and $\hat{f}$ vanish on the singular subset $\Sigma_{\mathrm{sig}}$.
\end{lemma}
Furthermore, we can show the constant terms appeared in Poisson summation of Theorem \ref{thm:possion summation} are all equal to zero. This follows from the next lemma. We use symbol $e(x)$ for $\exp 2 \pi \sqrt{-1} x$.
\begin{lemma}
With $\tilde{f} = f' \otimes f'' \in \mathcal{S}(V \oplus \mathbb{R}^2)$ chosen before, we have for $i \leq 3$,
\begin{equation*}
\begin{split}
d_{3, i}(\tilde{f})(0_{V_i}, 0, a) = 0 \quad \text{and } \quad d_{3, i}(F_{V}(\tilde{f}))(0_{V_i}, 0, a) = 0.
\end{split}
\end{equation*}
\end{lemma}
\begin{proof}
By definition of $d_{3, 2}$
\begin{equation*}
\begin{split}
d_{3, 2}(\tilde{f})(0_{V_2}, 0, a) = \int \tilde{f}(0_{V_2}, 0, u, 0, 0) e(u a) du = 0
\end{split}
\end{equation*}
because $(0_{V_2}, 0, u) \in \Sigma_{\text{sig}}$. Moreover, 
\begin{equation*}
\begin{split}
d_{3, 1}(\tilde{f})(0_{V_1}, 0, a) &= \iint \tilde{f}(0_{V_1}, 0, v, 0, u, 0, 0)  e(va) du dv = 0 \\
d_{3, 0}(\tilde{f})(0, a) &= \iiint \tilde{f}(0, w, 0, v, 0, u, 0, 0) e(wa) du dv dw = 0
\end{split}
\end{equation*}
for the same reason. 
The proof for $\mathcal{F}_V(\tilde{f})$ is similar.
\end{proof}
With this lemma, the Poisson summation formula can be simplified to only sum over quadrics. 
\begin{proposition}
\begin{equation*}
\label{prop:simplified Possion}
\begin{split}
\sum_{n=1}^{\infty} \sum_{\xi \in n X_i^0(\mathbb{Z})} n^{\frac{\mathrm{dim \ V_i}}{2} - 2} I(d_{3, i} (\tilde{f})) (\xi) = 
\sum_{n=1}^{\infty} \sum_{\xi \in n X_i^0(\mathbb{Z})} n^{\frac{\mathrm{dim \ V_i}}{2} - 2} I(d_{3, i} (  \mathcal{F}_{V} (\tilde{f}))) (\xi).
\end{split}
\end{equation*}    
\end{proposition}
Let $L' = L - L \cap \Sigma_{\mathrm{sig}}$ be the non-singular part of the lattice $L$. Following Shintani and in view of Proposition \ref{prop:summation at Archimedean place} ,  given a Schwartz function $f$ and its Fourier transform $\hat{f}$, zeta functions are defined by
\begin{equation*}
\begin{split}
Z(f, s_1, s_2, w) &:= \int_{B^{+} / B^{+}_{\mathbb{Z}}} \chi_1(h)^{s_1} \chi_2(h)^{s_2} \chi(h)^{w} \sum_{n=1}^{\infty} \sum_{v \in n L'} n^{\frac{\mathrm{dim}(V)}{2} -2}f(h \cdot v) dh, \\
Z(\hat{f}, s_1, s_2, w) &:= \int_{B^{+} / B^{+}_{\mathbb{Z}}} \chi_1(h)^{s_1} \chi_2(h)^{s_2} \chi(h)^{w} \sum_{n=1}^{\infty} \sum_{v \in n L'} n^{\frac{\mathrm{dim}(V)}{2} -2} \hat{f}(h \cdot v) dh
\end{split}
\end{equation*}
where $B^{+}_{\mathbb{Z}} = B^{+} \cap ( \mathrm{SL}_2(\mathbb{Z}) \times \mathrm{SL}_2(\mathbb{Z}) )$. 
Set
\begin{equation*}
\begin{split}
x_{+} = \left( \left( \begin{matrix}
    1    & 0\\
    0 &  -1
\end{matrix} \right), \left( \begin{matrix}
     1    & 0\\
    0 &   -1
\end{matrix} \right) \right), \quad x_{-} = \left( \left( \begin{matrix}
    1    & 0\\
    0 &  1
\end{matrix} \right), \left( \begin{matrix}
     1    & 0\\
    0 &   1
\end{matrix} \right) \right).
\end{split}
\end{equation*}
For each $f \in \mathcal{S}(X)$, define the orbital integral and the orbital integral associated with its Fourier transform
\begin{equation*}
\begin{split}
\Phi_{\pm}(f, s_1, s_2, w) &:= \int_{B^{+}} f(h \cdot x_{\pm})   \chi_1(h)^{s_1} \chi_2(h)^{s_2} \chi(h)^{w} dh, \\
\Phi_{\pm}(\hat{f}, s_1, s_2, w) &:= \int_{B^{+}} \hat{f}(h \cdot x_{\pm})   \chi_1(h)^{s_1} \chi_2(h)^{s_2} \chi(h)^{w} dh.
\end{split}
\end{equation*}
By Proposition \ref{prop:Dixmier-Malliavin}, there is $f' \in \mathcal{S}(V)$ such that $f'|_{X^0} = f$. By \cite[Lemma 3]{shintani1975zeta} these orbital integrals are absolutely convergent in the domain $\{(s_1, s_2, w): \mathrm{Re}(s_1) > 1, \mathrm{Re}(s_2) > 1,  \mathrm{Re}(w) > 2\}$. 
Next, define multiple Dirichlet series as follows
\begin{equation*}
\begin{split}
\xi_{+}(s_1, s_2, w) &:= \sum_{D > 0} \frac{1}{|D|^{w}} \sum_{n, m \geq 1} \frac{\sum_{\substack{d \geq 1, d^2|D \\ d|m, d|n}} d \cdot A(\frac{4m}{d}, \frac{D}{d^2}) A(\frac{4n}{d}, \frac{D}{d^2})}{m^{s_1} n^{s_2}}, \\
\xi_{-}(s_1, s_2, w) &:= \sum_{D < 0} \frac{1}{|D|^{w}} \sum_{n, m \geq 1} \frac{\sum_{\substack{d \geq 1, d^2|D \\ d|m, d|n}} d \cdot A(\frac{4m}{d}, \frac{D}{d^2}) A(\frac{4n}{d}, \frac{D}{d^2})}{m^{s_1} n^{s_2}}.
\end{split}
\end{equation*}
Then we have
\begin{proposition}
\label{prop:Z Phi relation}
Let $f$ and $\hat{f}$ be even functions. For $\mathrm{Re}(s_1) > 1, \mathrm{Re}(s_2) > 1, \mathrm{Re}(w) > 3$, the zeta functions $Z(f, s_1, s_2, w)$ and $Z(\hat{f}, s_1, s_2, w)$ are absolutely convergent  and they satisfy 
\begin{equation*}
\begin{split}
Z(f, s_1, s_2, w) & = 4^{w} \xi_{+}(s_1, s_2, w)\Phi_{+}(f, s_1, s_2, w) + 4^{w} \xi_{-}(s_1, s_2, w)\Phi_{-}(f, s_1, s_2, w) \\
Z(\hat{f}, s_1, s_2, w) & = 4^{w} \xi_{+}(s_1, s_2, w)\Phi_{+}(\hat{f}, s_1, s_2, w) + 4^{w} \xi_{-}(s_1, s_2, w)\Phi_{-}(\hat{f}, s_1, s_2, w).
\end{split}
\end{equation*}
\end{proposition}
\begin{proof}
For any $x \in B^{+} \cdot x_{+} \cap L$, there is a unique $h \in B^{+}$ such that $h x_{+} = x$.  Let $h_x$ denote this $h$. Write
\begin{equation*}
\begin{split}
x = \left( \left( \begin{matrix}
    m    & b/2\\
    b/2 &  c
\end{matrix} \right), \left( \begin{matrix}
     n    & b'/2\\
    b'/2 &  c'
\end{matrix} \right) \right), \quad D = b^2 - 4mc = (b')^2 - 4n c' > 0. 
\end{split}
\end{equation*}
Then 
\begin{equation*}
\begin{split}
(m, b, c)  = (\mu^2 t^2, 2\mu^2u, \mu^2 t^{-2}(u^2-1)), \quad  (n, b', c') = (\mu^2 s^2, 2\mu^2 v, \mu^2 s^{-2}(v^2-1)).
\end{split}
\end{equation*}
It yields
\begin{equation*}
\begin{split}
\chi_1(h_x) = m, \chi_2(h_x) = n, \chi(h_x) = D/4.
\end{split}
\end{equation*}

For the multiple Dirichlet series $\xi_{\pm}(s_1, s_2, w)$, the absolute convergence when $\mathrm{Re}(s_1) > 1$ follows from the fact that $A(\frac{4m}{d}, \frac{D}{d^2}) \leq A(4m, D)$ and the Dirichlet series $\sum_{m=1}^{\infty} A(4m, D) m^{-s_1}$ converges absolutely in the domain $\mathrm{Re}(s_1) > 1$. The reason for the absolute convergence when $\mathrm{Re}(s_2) > 1$ is the same. Since 
\begin{equation*}
\begin{split}
\sum_{\substack{d \geq 1, d^2|D \\ d|m, d|n}} d \cdot A(\frac{4m}{d}, \frac{D}{d^2})  A(\frac{4n}{d}, \frac{D}{d^2}) \leq \sigma_1(D) A(4m, D) A(4n, D),
\end{split}
\end{equation*}
and the series $\sum_{D = 1}^{\infty} A(4m, D) D^{-w}$, $\sum_{D = 1}^{\infty} A(4n, D) D^{-w}$  and $\sum_{D = 1}^{\infty} \sigma_1(D) D^{-w}$ are all absolutely convergent when $\mathrm{Re}(w) > 1$, it implies that the multiple Dirichlet series are absolutely convergent in the domain $\{(s_1, s_2, w): \mathrm{Re}(s_1) > 1, \mathrm{Re}(s_2) > 1,  \mathrm{Re}(w) > 3 \}$.

Therefore, unfolding the summation over lattice points, counting orbits and changing variables $h \to h h_x^{-1}$ we get
\begin{equation*}
\begin{split}
&\int_{B^{+} / B^{+}_{\mathbb{Z}}} \chi_1(h)^{s_1} \chi_2(h)^{s_2} \chi(h)^{w} \sum_{d=1}^{\infty} \sum_{x \in d \cdot L' \cap X^{+}} d^{\frac{\mathrm{dim} V}{2} -2} f(h \cdot x) dh \\
= & 4 \int_{B^{+}} \chi_1(h)^{s_1} \chi_2(h)^{s_2} \chi(h)^{w} \sum_{d= 1}^{\infty}\frac{1}{4} \frac{4^{w}}{D^{w}} \sum_{\substack{m, n , D \geq 1,  \\ d^2|D, d|m, d|n}} \frac{d \cdot A(\frac{4m}{d}, \frac{D}{d^2}) A(\frac{4n}{d}, \frac{D}{d^2})}{m^{s_1} n^{s_2}}  f(h \cdot x_{+})  dh \\
= & 4^{w} \sum_{m, n, D = 1}^{\infty} \sum_{\substack{d \geq 1, d^2|D \\ d|m, d|n}} \frac{d \cdot A(\frac{4m}{d}, \frac{D}{d^2}) A(\frac{4n}{d}, \frac{D}{d^2})}{m^{s_1} n^{s_2} D^{w}} \int_{B^{+}} \chi_1(h)^{s_1} \chi_2(h)^{s_2} \chi(h)^{w}  f(h \cdot x_{+})  dh.
\end{split}
\end{equation*}

\end{proof}
Define
\begin{equation*}
\begin{split}
Z_{+}(f, s_1, s_2, w) &:= \int_{B^{+} / B^{+}_{\mathbb{Z}}, \chi(h)\geq 1} \chi_1(h)^{s_1} \chi_2(h)^{s_2} \chi(h)^{w} \sum_{n=1}^{\infty} \sum_{v \in n L'} n^{\frac{\mathrm{dim} V}{2} -2} f(h \cdot v) dh, \\
Z_{+}(\hat{f}, s_1, s_2, w) &:= \int_{B^{+} / B^{+}_{\mathbb{Z}}, \chi(h)\geq 1} \chi_1(h)^{s_1} \chi_2(h)^{s_2} \chi(h)^{w} \sum_{n=1}^{\infty} \sum_{v \in n L'} n^{\frac{\mathrm{dim} V}{2} -2} \hat{f} (h \cdot v) dh.
\end{split}
\end{equation*}
From Proposition \ref{prop:Z Phi relation} it follows that both $Z_{+}(f, s_1, s_2, w)$ and $Z_{+}(\hat{f}, s_1, s_2, w)$ are absolutely convergent in the domain $\{(s_1, s_2, w): \mathrm{Re}(s_1) > 1, \mathrm{Re}(s_2) > 1, w \in \mathbb{C}\}$.
\begin{lemma}
\label{lemma:changing variable from h to h^-1}
For $\mathrm{Re}(s_1) > 1$ and $\mathrm{Re}(s_2) > 1$, 
\begin{equation*}
\begin{split}
\int_{B^{+} / B^{+}_{\mathbb{Z}}, \chi(h)\leq 1} \chi_1(h)^{s_1} \chi_2(h)^{s_2} \chi(h)^{w} \sum_{L'} f( \chi^{-1}(h)h \cdot  v) dh = Z_{+}(f, s_1, s_2, -s_1-s_2-w)
\end{split}
\end{equation*}
where $\chi(h)^{-1} h$ denotes the product $\chi(h)^{-1}I_V \cdot \iota(h)$.
\end{lemma}
\begin{proof}
After changing variable $\mu \to \mu^{-1}$ in  the Iwasawa decomposition of $h = \left( \begin{matrix}
    d_{\mu} & \\
        &d_{\mu}
\end{matrix}  \right) \left( \begin{matrix}
    a_t & \\
        &a_{s}
\end{matrix}  \right) \left( \begin{matrix}
    n_u & \\
        &n_v
\end{matrix}  \right)$, it gives
\begin{equation*}
\begin{split}
\chi_1(h) \to \chi_1(h) \chi(h)^{-1}, \chi_2(h) \to \chi_2(h) \chi(h)^{-1}, \chi(h)^{-1} h \to h.
\end{split}
\end{equation*}
Then the result follows.
\end{proof}

\subsection{Functional equation of zeta functions and  multiple Dirichlet series}
Now we are ready to prove the main result of this paper.
\begin{theorem}
\label{thm:Z functional equation}
For $\mathrm{Re}(s_1), \mathrm{Re}(s_2) > 1$,  the zeta functions $Z(f, s_1, s_2, w)$ and $Z(\hat{f}, s_1, s_2, w)$ can be analytically continued as entire functions in the whole $w$-plane which satisfy the following functional equation
\begin{equation*}
\begin{split}
Z( f, s_1, s_2, w)  &=  Z(\hat{f}, s_1, s_2, 2-s_1-s_2-w).
\end{split}
\end{equation*}
\end{theorem}
\begin{proof}
Expressing the integration in terms of Iwasawa decomposition for $B^{+}$,  the global zeta functions $Z(f, s_1, s_2, w)$ becomes
\begin{equation*}
\begin{split}
& \int_{A^{+}} \int_{\Lambda^{+} N / N_{\mathbb{Z}} }  \chi_1(a)^{s_1-\frac{1}{2}} \chi_2(a)^{s_2-\frac{1}{2}} \chi(h)^{\frac{s_1+s_2+2w}{2}} \sum_{n=1}^{\infty} \sum_{v \in n L'} n^{\frac{\mathrm{dim} V}{2} -2}  f(an h \cdot v) dh dn da.
\end{split}
\end{equation*}
By Poisson summation formula simplified in Proposition \ref{prop:simplified Possion}, 
\begin{equation*}
\begin{split}
\sum_{n=1}^{\infty} \sum_{v \in n X_i^0(\mathbb{Z})} n^{\frac{\mathrm{dim \ V_i}}{2} - 2} I(d_{3, i} (\tilde{f})) (v) = 
\sum_{n=1}^{\infty} \sum_{v\in n X_i^0(\mathbb{Z})} n^{\frac{\mathrm{dim \ V_i}}{2} - 2} I(d_{3, i} (  \mathcal{F}_{V} (\tilde{f}))) (v),
\end{split}
\end{equation*}    
we calculate the inner integration as follows
\begin{equation*}
\begin{split}
 & \int_{\Lambda^{+} N / N_{\mathbb{Z}}} \chi(h)^{\frac{s_1+s_2+2(w+1)}{2}} \sum_{v \in n L'} n^{\frac{\mathrm{dim  V}}{2} - 2} I (\sigma_3(anh) \tilde{f}) (v) dh dn \\
& = \int_{\substack{\Lambda^{+} N / N_{\mathbb{Z}}\\ \chi(h) \geq 1}} \chi(h)^{\frac{s_1+s_2+2(w+1)}{2}} \sum_{v \in n L'} n^{\frac{\mathrm{dim  V}}{2} - 2} I (\sigma_3(anh) \tilde{f}) (v) dh dn \\
&+ \int_{\substack{\Lambda^{+} N / N_{\mathbb{Z}}\\ \chi(h) \leq 1}} \chi(h)^{\frac{s_1+s_2+2(w+1)}{2}} \sum_{v \in n L'} n^{\frac{\mathrm{dim  V}}{2} - 2} I (\mathcal{F}_V(\sigma_3(anh) \tilde{f})) (v) dh dn\\
& + \int_{\substack{\Lambda^{+} N / N_{\mathbb{Z}}\\ \chi(h) \leq 1}} \chi(h)^{\frac{s_1+s_2+2(w+1)}{2}} \sum_{i=1}^2 \sum_{v \in n X_i^0} n^{\frac{\mathrm{dim  V_i}}{2} - 2} I(d_{3, i} (\mathcal{F}_V(\sigma_3(anh) \tilde{f}))) (v)  dh dn \\
& - \int_{\substack{\Lambda^{+} N / N_{\mathbb{Z}}\\ \chi(h) \leq 1}} \chi(h)^{\frac{s_1+s_2+2(w+1)}{2}} \sum_{i=1}^2 \sum_{v \in n X_i^0} n^{\frac{\mathrm{dim  V_i}}{2} - 2} 
 I(d_{3, i} (\sigma_3(anh) \tilde{f})) (v) dh dn.
\end{split}
\end{equation*}    
By Proposition \ref{prop:vanishing nilpotent}, the contribution to the integration from summations over lower-dimensional quadrics is zero. 
Hence, by Proposition \ref{prop:h}, 
\begin{equation*}
\begin{split}
 & \int_{\Lambda^{+} N / N_{\mathbb{Z}}} \chi(h)^{\frac{s_1+s_2+2w}{2}} \sum_{v \in n L'} n^{\frac{\mathrm{dim  V}}{2} - 2} f(anh \cdot v) dh dn \\
& = \int_{\substack{\Lambda^{+} N / N_{\mathbb{Z}}\\ \chi(h) \geq 1}} \chi(h)^{\frac{s_1+s_2+2w}{2}} \sum_{v \in n L'} n^{\frac{\mathrm{dim  V}}{2} - 2} f (anh \cdot v) dh dn\\
&+ \int_{\substack{\Lambda^{+} N / N_{\mathbb{Z}}\\ \chi(h) \geq 1}} \chi(h)^{\frac{-s_1-s_2+2(2 -w)}{2}} \sum_{v \in n L'} n^{\frac{\mathrm{dim  V}}{2} - 2} \hat{f}(a n h\cdot v) dh dn
\end{split}
\end{equation*}    
where the last integral follows from changing the variable $h \to h^{-1}$ as we did in the proof of Lemma \ref{lemma:changing variable from h to h^-1}.
Therefore we get
\begin{equation}
\label{eq:Z entire function}
\begin{split}
Z( f, s_1, s_2, w)  &=  Z_{+}( f, s_1, s_2, w) +  Z_{+}(\hat{f}, s_1, s_2, 2-s_1-s_2-w).
\end{split}
\end{equation}
which in turn implies 
both $Z( f, s_1, s_2, w)$ and $Z(\hat{f}, s_1, s_2, w)$ can be analytically continued as entire functions in the whole $w$-plane which satisfy the functional equation
\begin{equation*}
\begin{split}
Z( f, s_1, s_2, w)  &=  Z(\hat{f}, s_1, s_2, 2-s_1-s_2-w).
\end{split}
\end{equation*}
\end{proof}
%\begin{comment}
In \cite{shintani1975zeta}, Shintani showed the double Dirichlet series \eqref{eq:double series} can be completed such that
\begin{equation*}
\begin{split}
(2 \pi)^{-s} \left(\sin \frac{\pi s}{2} \right)^{-1} \Gamma(s)\zeta(2 s) \xi_1(s, w) \quad \mathrm{and} \quad (2 \pi)^{-s} \Gamma(s)\zeta(2 s) \xi_2(s, w)
\end{split}
\end{equation*}
are all invariant under the transformation $(s, w) \to (1-s, s + w - \frac{1}{2})$. We will prove the similar invariance property for  multiple Dirichlet series $\xi_{\pm}(s_1, s_2, w)$. For this we define the usual subgroups of $\mathrm{SL}_2(\mathbb{R})$ as follows 
\begin{equation*}
\begin{split}
K &= \left\{ \left( \begin{matrix}
    \cos \theta & \sin \theta\\
     - \sin \theta   &\cos \theta
\end{matrix}  \right): \theta \in \mathbb{R} \right\}, \\
A_2^{+} &= \left\{ \left( \begin{matrix}
    t & \\
        &t^{-1}
\end{matrix}  \right): t > 0 \right\}, \\
N_2 &= \left\{ \left( \begin{matrix}
    1 & u \\
     0   &1
\end{matrix}  \right): u \in \mathbb{R} \right\}, \\
N_2' &= \left\{ \left( \begin{matrix}
    1 & 0 \\
     u   &1
\end{matrix}  \right): u \in \mathbb{R} \right\}.
\end{split}
\end{equation*}

\begin{theorem}
The multiple Dirichlet series $\xi_{\pm}(s_1, s_2, w)$
have analytic continuations as meromorphic functions in $\mathbb{C}^3$. In addition to the functional equation given in Theorem \ref{thm:Z functional equation},   the functions
\begin{equation}
\begin{split}
&\prod_{i=1}^2 (2 \pi)^{-s_i} \left(\sin \frac{\pi s_i}{2} \right)^{-1} \Gamma(s_i)\zeta(2 s_i) \xi_{+}(s_1, s_2, w),\\
&\prod_{i=1}^2  (2 \pi)^{-s_i} \Gamma(s_i)\zeta(2 s_i) \xi_{-}(s_1, s_2, w)
\end{split}
\end{equation}
are invariant under the transformations $(s_i, w) \to (1-s_i, s_i + w - \frac{1}{2})$ for $i = 1, 2$.
\end{theorem}
\begin{proof}
We follow Shintani's convention in \cite{shintani1972dirichlet} to define Eisenstein series on $\mathrm{SL}_2(\mathbb{R})$. The Iwasawa decomposition is $K A^{+}_2 N'_2$ with Haar measure $dg =  du \frac{d^{\times}t}{t^2} d\theta$. We write $t(g)$ for the elements $A^{+}_2$ in the representation of $g$ in the Iwasawa decomposition. The Eisenstein series $E(z, g)$ for $\mathrm{Re}(z) > 1$ is defined by
\begin{equation*}
\begin{split}
E(z, g) = \frac{1}{2} \sum_{\mathrm{SL_2}(\mathbb{Z} )\backslash  B_2^{+}(\mathbb{Z}) } t(g \gamma)^{z + 1}
\end{split}
\end{equation*}
which has an analytic continuation to a meromorphic function on the whole complex plane. The functional equation of $E(z, g)$ is given by
\begin{equation*}
\begin{split}
\xi(1 + z) E(g, z) = \xi(1 - z) E(g, -z)
\end{split}
\end{equation*}
where $\xi(z) = \pi^{-\frac{z}{2}} \Gamma\left(\frac{z}{2}  \right) \zeta(z)$. 

Let $B^1_2(\mathbb{R})$ denote $B^{+}_2(\mathbb{R})\cap \mathrm{SL}_2(\mathbb{R})$. We write $[ \mathrm{SL}_2 ]$ for $\mathrm{SL_2}(\mathbb{R} )\backslash  \mathrm{SL_2}(\mathbb{Z})$ and $[B^{1}_2]$ for $B^1_2(\mathbb{R}) \backslash B^{+}_2(\mathbb{Z})$. Then for any  $\mathrm{SO}_2$-invariant Schwartz function $f \in \mathcal{S}(X)$ such that $f$ vanishes on the singular subset $\Sigma_{\mathrm{sig}}$, the zeta function $Z(f, s_1, s_2, w)$ can be written as, after unfolding the Eisenstein series $E(z, g)$,
\begin{multline*}
Z(f, s_1, s_2, w)  = \pi^{-1} \int_{0}^{\infty} \int_{h_2 \in [B^1_2]}  \int_{h_1 \in [\mathrm{SL}_2]}  E(2s_1-1, h_1) \chi_2(h_2)^{s_2} \lambda^{4(w + \frac{s_1}{2} + \frac{s_2}{2})}  \\
 \times \sum_{n=1}^{\infty} \sum_{v \in n L'} n^{\frac{\mathrm{dim}(V)}{2} -2}f((h_1, h_2) \cdot \lambda^2 v) dh_1 dh_2 d^{\times}\lambda.
\end{multline*}
By Proposition \ref{prop:Z Phi relation}, the domain of absolute convergence is given by $\{ (s_1, w, s_2): \mathrm{Re}(s_1) > 1, \mathrm{Re}(w) > 3, \mathrm{Re}(s_2) > 1 \}$.  By functional equation of $E(z, g)$, this domain can be extended to $\{ (s_1, w, s_2): \mathrm{Re}(s_1) < 0, \mathrm{Re}(s_1 + 2w) > 6 + (1 - \mathrm{Re}(s_1)), \mathrm{Re}(s_2) > 1 \}$. By Bochner's convexity theorem \cite{bochner1938theorem}, it can be analytically continued to $\{ (s_1, w, s_2): 0 < \mathrm{Re}(s_1) < 1, \mathrm{Re}(s_1 + 2w) > 6 + 1, \mathrm{Re}(s_2) > 1 \}$. Therefore $Z(f, s_1, s_2, w)$ has analytic continuation to a meromorphic function in the domain $\{ (s_1, w, s_2): \mathrm{Re}(s_1 + 2w) > 6 + \mathrm{max}(\mathrm{Re}(s_1), 1-\mathrm{Re}(s_1), 1 ), \mathrm{Re}(s_2) > 1 \}$. Furthermore,  $Z_{+}(f, s_1, s_2, w)$ has analytic continuation to a  meromorphic function in the domain $\{(s_1, w, s_2): (s_1, w) \in \mathbb{C}^2,  \mathrm{Re}(s_2) > 1 \}$. The function 
\begin{equation*}
\begin{split}
\pi^{-s_1} \Gamma\left( s_1 \right) \zeta(2s_1) E(2s_1-1, h_1) \lambda^{4 (w + \frac{s_1}{2})}
\end{split}
\end{equation*}
is invariant under the transformation $(s_1, w) \to (1-s_1, s_1 + w - \frac{1}{2})$. Hence 
\begin{equation*}
\begin{split}
\pi^{-s_1} \Gamma\left( s_1 \right) \zeta(2s_1) Z(f, s_1, s_2, w), \quad \pi^{-s_1} \Gamma\left( s_1 \right) \zeta(2s_1) Z_{+}(f, s_1, s_2, w) 
\end{split}
\end{equation*}
are also invariant under the transformation $(s_1, w) \to (1-s_1, s_1 + w - \frac{1}{2})$. By symmetry of $s_1, s_2$, combining the equation \eqref{eq:Z entire function} shows $Z(f, s_1, s_2, w)$ has analytic continuation to a meromorphic function in the entire domain $\{ (s_1, w, s_2) \in \mathbb{C}^3 \}$.

Now we prove the invariance property for the orbital integrals
$\Phi_{\pm}(f, s_1, s_2, w)$.  Write $(x_{\pm, 1}, x_{\pm, 2})$ for $x_{\pm}$. Let $V_i = V_{\mathrm{bqf}, i}(\mathbb{R}) \text{ (i = 1, 2)}$.  By Proposition \ref{prop:Dixmier-Malliavin} , we can choose $f_1 \otimes f_2 \in  \mathcal{S}(V_{ 1}) \otimes \mathcal{S}(V_{ 2}) \subset \mathcal{S}(V)$ such that $f = \left(f_1 \otimes f_2\right)|_{X^0}$. Then the orbital integrals can be written as
\begin{multline*}
\Phi_{\pm}(f, s_1, s_2, w)  = \int_{0}^{\infty} \int_{h_1 \in B^1_2(\mathbb{R})}  \int_{h_2 \in B^1_2(\mathbb{R})}  \chi_1(h_1)^{s_1} \chi_2(h_2)^{s_2} \lambda^{4(w + \frac{s_1}{2} + \frac{s_2}{2})}  \\
\times f_1( h_1 \cdot \lambda^2 x_{\pm, 1}) f_2( h_2 \cdot \lambda^2 x_{\pm, 2}) dh_1 dh_2 d^{\times}\lambda.
\end{multline*}
We identity $V_{\pm, 1} = B_2^{+} \cdot x_{\pm}$ which is open dense in $\mathbb{R}^3$. Let $dx = dx_1 dx_2 dx_3$ denote the usual volume form on $\mathbb{R}^3$. Then the action $g = d_{\lambda} a_t n_u$ on $x_{\pm}$ gives the change of coordinates such that the volume form is equal to
\begin{equation*}
\begin{split}
dx_1 dx_2 dx_3 = 16 (\lambda t)^2  \lambda^4 d^{\times}\lambda du \frac{d^{\times}t}{t^2} = 16 \chi_1(g) \chi(g) dg. 
\end{split}
\end{equation*}
Hence 
\begin{multline*}
\Phi_{\pm}(f, s_1, s_2, w)  = 4^{-(w+\frac{s_2}{2}+1)} \int_{x \in V_{\pm, 1}} \int_{h_2 \in B^1_2(\mathbb{R})}  |P_1(x)|^{s_1-1}  |P(x)|^{w + \frac{s_2}{2} -1 } \chi_2(h_2)^{s_2}  \\
\times f_1( x ) f_2( h_2 \cdot 2^{-1}\sqrt{|P(x)|} x_{\pm, 2}) dx dh_2.
\end{multline*}
For $(\theta, u, t) \in [0, \pi] \times \mathbb{R} \times \mathbb{R}_{+}$, change variables 
\begin{equation*}
\begin{split}
x(\theta, u, t)  = k_{\theta} \cdot (u, \sqrt{4 u^2 + t}, u)
\end{split} \quad \text{if $x \in V_{+}$},
\end{equation*}
and
\begin{equation*}
\begin{split}
x(\theta, u, t)  = k_{\theta} \cdot (u, \sqrt{4 u^2 - t}, u)
\end{split} \quad \text{$4u^2 \geq t$  if $x \in V_{-}$}.
\end{equation*}
Assume $f_1$ is $\mathrm{SO}_2$-invariant. Then 
\begin{multline*}
     \int_{x \in V_{+, 1}} |P_1(x)|^{s_1-1}  |P(x)|^{w -1 }  f_1( x ) f_2( h_2 \cdot 2^{-1}\sqrt{|P(x)|} x_{+, 2}) dx  \\
     = 2^{-1}\int_{0}^{\infty} \int_{-\infty}^{\infty} \int_{0}^{\pi} |u + \sqrt{u^2 + \frac{t}{4}} \sin(2 \theta)|^{s_1-1}  t^{w -1 }  f_1( x(0, u, t) )   \\ 
\times f_2( h_2 \cdot 2^{-1}\sqrt{t} x_{+, 2}) d\theta du dt ,
\end{multline*}
and 
\begin{multline*}
     \int_{x \in V_{-, 1}} |P_1(x)|^{s_1-1}  |P(x)|^{w -1 }  f_1( x ) f_2( h_2 \cdot 2^{-1}\sqrt{|P(x)|} x_{-, 2}) dx  \\
     = \int_{0}^{\infty} \int_{2^{-1}\sqrt{t}}^{\infty} \int_{0}^{\pi} |u + \sqrt{u^2 - \frac{t}{4}} \sin(2 \theta)|^{s_1-1}  t^{w -1 }  f_1( x(0, u, t) )   \\ 
\times f_2( h_2 \cdot 2^{-1}\sqrt{t} x_{-, 2}) d\theta du dt ,
\end{multline*}
Using integral representation of  the Legendre function we have
\begin{multline*}
\int_{0}^{\pi} |u + \sqrt{u^2 + \frac{t}{4}} \sin(2 \theta)|^{s_1-1} d\theta = 2^{-s_1+1} t^{\frac{s_1-1}{2}} \frac{\pi}{2} \left(\cos \frac{\pi (s_1-1)}{2}  \right)^{-1} \\
\times \left( P_{s_1-1}(2 \sqrt{-1} u t^{-1/2} ) + P_{s_1-1}(-2 \sqrt{-1} u t^{-1/2} ) \right),
\end{multline*}
and
\begin{align*}
\int_{0}^{\pi} |u + \sqrt{u^2 - \frac{t}{4}} \sin(2 \theta)|^{s_1-1} d\theta = 2^{-s_1+1} t^{\frac{s_1-1}{2}} \pi  
 P_{s_1-1}(2 u t^{-1/2} ) \quad \text{for $ u > 2^{-1} \sqrt{t}$} .
\end{align*}
The Legendre function $P_s(z)$ is an entire function of $s$ if $z > 1$ or $z \in i \mathbb{R}$ and satisfies the functional equation $P_{s-1}(z) = P_{-s}(z)$.  It follows that
\begin{align*}
4^w 2^{s_1} \sin\left( \frac{\pi s_1}{2} \right) \Phi_{+}(f_1, s_1, s_2, w),  \quad  4^w2^{s_1}  \Phi_{-}(f_1, s_1, s_2, w)
\end{align*}
are invariant under the transformation $(s_1, w) \to (1-s_1, s_1 + w - \frac{1}{2})$.  Hence Proposition \ref{prop:Z Phi relation} implies that 
\begin{equation*}
\begin{split}
(2 \pi)^{-s_1} \left(\sin \frac{\pi s_1}{2} \right)^{-1}\Gamma(s_1)\zeta(2 s_1) \xi_{+}(s_1, s_2, w), \quad (2 \pi)^{-s_1} \Gamma(s_1)\zeta(2 s_1) \xi_{-}(s_1, s_2, w).
\end{split}
\end{equation*}
are also invariant under the transformation $(s_1, w) \to (1-s_1, s_1 + w - \frac{1}{2})$ and have analytic continuations to meromorphic functions in the domain $\{(s_1, w, s_2): (s_1, w) \in \mathbb{C}^2,  \mathrm{Re}(s_2) > 1 \}$. Now the assertion of the theorem follows from the symmetry of $s_1$ and $s_2$.  
\end{proof}

%\end{comment}

\printbibliography %Prints bibliography

@misc{getz2023summation,
      title={Summation formulae for quadrics}, 
      author={Jayce R. Getz},
      year={2023},
      eprint={2201.02583},
      archivePrefix={arXiv},
      primaryClass={math.NT}
}

@article{shintani1975zeta,
  title={On zeta functions associated with the vector space of quadratic forms},
  author={Shintani, Takuro},
  journal={J. Fac. Sci. Univ. Tokyo},
  volume={22},
  pages={26--65},
  year={1975}
}

@article{shintani1972dirichlet,
  title={On Dirichlet series whose coefficients are class numbers of integral binary cubic forms},
  author={Shintani, Takuro},
  journal={Journal of the Mathematical Society of Japan},
  volume={24},
  number={1},
  pages={132--188},
  year={1972},
  publisher={The Mathematical Society of Japan}
}

@article{sato1974zeta,
  title={On zeta functions associated with prehomogeneous vector spaces},
  author={Sato, Mikio and Shintani, Takuro},
  journal={Annals of Mathematics},
  volume={100},
  number={1},
  pages={131--170},
  year={1974},
  publisher={JSTOR}
}

@article{diamantis2014converse,
  title={A converse theorem for double Dirichlet series and Shintani zeta functions},
  author={Diamantis, Nikolaos and Goldfeld, Dorian},
  journal={Journal of the Mathematical Society of Japan},
  volume={66},
  number={2},
  pages={449--477},
  year={2014},
  publisher={The Mathematical Society of Japan}
}

@inproceedings{kim2022shintani,
  title={The Shintani double zeta functions},
  author={Kim, Henry H and Tsuzuki, Masao and Wakatsuki, Satoshi},
  booktitle={Forum Mathematicum},
  volume={34},
  number={2},
  pages={469--505},
  year={2022},
  organization={De Gruyter}
}

@incollection{braverman2010gamma,
  title={$\gamma$-functions of representations and lifting},
  author={Braverman, Alexander and Kazhdan, David},
  booktitle={Visions in Mathematics: GAFA 2000 Special Volume, Part I},
  pages={237--278},
  year={2010},
  publisher={Springer}
}

@article{braverman2002normalized,
  title={Normalized intertwining operators and nilpotent elements in the Langlands dual group},
  author={Braverman, Alexander and Kazhdan, David},
  journal={Mosc. Math. J.},
  volume={2},
  number={3},
  pages={533--553},
  year={2002}
}

@article{braverman1999schwartz,
  title={On the Schwartz space of the basic affine space},
  author={Braverman, Alexander and Kazhdan, David},
  journal={Selecta Mathematica},
  volume={5},
  number={1},
  pages={1},
  year={1999},
  publisher={Springer}
}

@article{ngo2020hankel,
  title={Hankel transform, Langlands functoriality and functional equation of automorphic L-functions},
  author={Ng{\^o}, Bao Ch{\^a}u},
  journal={Japanese Journal of Mathematics},
  volume={15},
  pages={121--167},
  year={2020},
  publisher={Springer}
}

@article{sakellaridis2012spherical,
  title={Spherical varieties and integral representations of L-functions},
  author={Sakellaridis, Yiannis},
  journal={Algebra \& Number Theory},
  volume={6},
  number={4},
  pages={611--667},
  year={2012},
  publisher={Mathematical Sciences Publishers}
}

@article{lafforgue2014noyaux,
  title={Noyaux du transfert automorphe de Langlands et formules de Poisson non lin{\'e}aires},
  author={Lafforgue, Laurent},
  journal={Japanese Journal of Mathematics},
  volume={9},
  pages={1--68},
  year={2014},
  publisher={Springer}
}

@article{getz2019summation,
  title={A summation formula for triples of quadratic spaces},
  author={Getz, Jayce R and Liu, Baiying},
  journal={Advances in Mathematics},
  volume={347},
  pages={150--191},
  year={2019},
  publisher={Elsevier}
}

@article{getz2020summation,
  title={A summation formula for triples of quadratic spaces II. arXiv e-prints, page},
  author={Getz, Jayce R and Hsu, Chun-Hsien},
  journal={arXiv preprint arXiv:2009.11490},
  volume={2},
  year={2020}
}

@article{getz2021refined,
  title={A refined Poisson summation formula for certain Braverman-Kazhdan spaces},
  author={Getz, Jayce Robert and Liu, Baiying},
  journal={Science China Mathematics},
  volume={64},
  pages={1127--1156},
  year={2021},
  publisher={Springer}
}

@article{getz2021harmonic,
  title={Harmonic analysis on certain spherical varieties},
  author={Getz, Jayce R and Hsu, Chun-Hsien and Leslie, Spencer},
  journal={arXiv preprint arXiv:2103.10261},
  year={2021}
}

@article{bhargava2004higher,
  title={Higher composition laws I: A new view on Gauss composition, and quadratic generalizations},
  author={Bhargava, Manjul},
  journal={Annals of Mathematics},
  volume={159},
  number={1},
  pages={217--250},
  year={2004},
  publisher={JSTOR}
}

@inproceedings{taniguchi2007zeta,
  title={On the zeta functions of prehomogeneous vector spaces for a pair of simple algebras},
  author={Taniguchi, Takashi},
  booktitle={Annales de l'institut Fourier},
  volume={57},
  number={4},
  pages={1331--1358},
  year={2007}
}

@article{chinta2007weyl,
  title={Weyl group multiple Dirichlet series constructed from quadratic characters},
  author={Chinta, Gautam and Gunnells, Paul E},
  journal={Inventiones mathematicae},
  volume={167},
  pages={327--353},
  year={2007},
  publisher={Springer}
}

@article{wen2013bhargava,
  title={Bhargava integer cubes and Weyl group multiple Dirichlet series},
  author={Wen, Jun},
  journal={arXiv preprint arXiv:1311.2132},
  year={2013}
}

@article{bochner1938theorem,
  title={A theorem on analytic continuation of functions in several variables},
  author={Bochner, Salomon},
  journal={Annals of Mathematics},
  pages={14--19},
  year={1938},
  publisher={JSTOR}
}

\end{document}